\documentclass{amsart}

\usepackage{amssymb, amsthm}

\allowdisplaybreaks[4]

\newcommand{\pD}[2]{\frac{\partial #1}{\partial #2}}

\newcommand{\rD}[2]{\frac{d #1}{d #2}}

\newcommand{\vn}[1]{\lVert#1\rVert}

\newcommand{\IP}[2]{\left< #1 , #2 \right>}
\newcommand{\eIP}[2]{\left(\left. #1 \right| #2 \right)}
\newcommand{\sfrac}[2]{\text{\fontsize{5}{5}\selectfont$\frac{#1}{#2}$}}

\newcommand{\R}{\ensuremath{\mathbb{R}}}
\newcommand{\N}{\ensuremath{\mathbb{N}}}

\newcommand{\hH}{\ensuremath{h_{|H|}}}
\newcommand{\DP}{\ensuremath{\delta^{(p)}}}

\newtheorem{thm}{Theorem}
\newtheorem{cor}[thm]{Corollary}
\newtheorem{prop}[thm]{Proposition}
\newtheorem{lem}[thm]{Lemma}

\theoremstyle{remark}
\newtheorem*{rmk}{Remark}

\begin{document}

\title{Lifespan theorem for constrained surface diffusion flows}
\author{James McCoy
   \and Glen Wheeler$^*$
   \and Graham Williams}
\thanks{* Corresponding author, \texttt{gew75@uow.edu.au}}
\address{School of Mathematics and Applied Statistics\\
         University of Wollongong\\
         Northfields Ave\\
         Wollongong, NSW 2500}

\begin{abstract}
We consider closed immersed hypersurfaces in $\R^{3}$ and $\R^4$ evolving by a class of constrained
surface diffusion flows.  Our result, similar to earlier results for the Willmore flow, gives both a
positive lower bound on the time for which a smooth solution exists, and a small upper bound on a
power of the total curvature during this time.  By phrasing the theorem in terms of the
concentration of curvature in the initial surface, our result holds for very general initial data
and has applications to further development in asymptotic analysis for these flows.
\end{abstract}

\maketitle


%

\section{Introduction.}

Let $f:M^n\times[0,T)\rightarrow\R^{n+1}$ be a family of compact immersed hypersurfaces $f(\cdot,t)
= f_t: M \rightarrow f_t(M) = M_t$ with associated Laplace-Beltrami operator $\Delta$, unit normal vector
field $\nu$, and mean curvature function $H$.  The surface diffusion flow
\begin{equation}
\label{SD}
\tag{SD}
\pD{}{t}f = (\Delta H)\nu,
\end{equation}
and the more general constrained surface diffusion flows
\begin{equation}
\tag{CSD}
\label{CSD}
\pD{}{t}f = (\Delta H + h)\nu,
\end{equation}
where $h:I\rightarrow\R$ and $I\supset[0,T)$, are the chief objects of interest for this paper.  Our
aim is to begin a systematic study of the regularity of the flows \eqref{CSD}.  We are motivated
chiefly by the examples
\begin{align*}
h\equiv 0,\ 
h_H &= \frac{\int_M \vn{\nabla H}^2 d\mu}{\int_M H d\mu}
,\ 
\hH = \frac{\int_M \vn{\nabla H}^2 d\mu}{\int_M |H| d\mu}
\text{,\ and}\ 
h_K = \frac{-\int_M (\Delta H)K d\mu}{\int_M K d\mu},
\end{align*}
where $K$ is the Gauss curvature of $M_t$.

The first is simply surface diffusion flow \eqref{SD}.  Using $\text{Vol }M_t$ to denote the volume
enclosed by $M_t$ in $\R^{n+1}$ we compute
\begin{align*}
\rD{}{t}\text{Vol }M_t &= \int_{M} \Delta Hd\mu = 0
\text{, and}\\
\rD{}{t}\int_{M} d\mu &= \int_{M} H\Delta H d\mu = -\int_M \vn{\nabla H}^2 d\mu \le 0;
\end{align*}
so that a manifold evolving by \eqref{SD} will exhibit conservation of enclosed volume and monotonic
decreasing surface area.  Further, surface area is preserved exactly when the mean curvature of
$M_t$ is constant.  It is these geometric characteristics of the surface diffusion flow which
motivate the generalisation to constrained surface diffusion flows.  For example with $h=h_H$, while
$\int_M Hd\mu\ne0$ we have
\begin{align*}
\rD{}{t}\int_{M} d\mu &= \int_{M} H\Delta H d\mu + h_H\int_M H d\mu \\
                      &= - \int_{M} \vn{\nabla H}^2 d\mu
                        + \frac{\int_M \vn{\nabla H}^2 d\mu}{\int_M H d\mu}\int_M H d\mu = 0,
\end{align*}
and now surface area is conserved.  Volume is monotonic increasing or decreasing depending on 
the sign of $\int_M H d\mu$, and preserved only when $H$ is constant.  
Unfortunately, it seems more difficult to show that qualities such as convexity are preserved.
This is due to the
absence of a maximum principle.  In particular, $\int_M H d\mu$ could approach zero under
\eqref{CSD} with $h=h_H$, which would cause the flow to be undefined, and most likely \emph{without}
a curvature singularity.  This motivates the use of $\hH$, where we replace the denominator with
total mean curvature $\int_M |H|d\mu$.  For this flow we compute
\begin{align*}
\rD{}{t}\text{Vol }M_t &= \int_{M} (\Delta H + \hH) d\mu
                        = |M|\frac{\int_M \vn{\nabla H}^2 d\mu}{\int_M |H| d\mu} \ge 0
\text{,\ \ \ and}\\
\rD{}{t}\int_{M} d\mu &= \int_{M} H(\Delta H + \hH) d\mu\\
                      &= -\int_M \vn{\nabla H}^2 d\mu + \int_M \vn{\nabla H}^2 d\mu
                          \frac{\int_M H d\mu}{\int_M |H| d\mu}
                      \le 0.
\end{align*}
Here enclosed volume and surface area are monotonic increasing and decreasing respectively.
We also have not only that surface area is stationary
(constant in time) if $H$ is constant, but volume also.  
It can also be observed that if volume is constant, then surface area is necessarily constant.  This
is in contrast to \eqref{SD} flow, where volume is constant regardless of the behaviour of the
surface area.
Further, the flow speed itself is non-zero for surfaces of piecewise linear mean curvature.  This
leads us to believe that singularity development and asymptotic behaviour under \eqref{CSD} flow
with $h=\hH$ will be easier to understand compared with that of \eqref{SD} flow.  (Consider for
example a clothoid-type manifold.)  Finally, we use an inequality of
Burago-Zalgaller \cite{burago1988gi} to infer
\[
\int_M |H| d\mu \ge c_{B\!Z}|M_t|^{\frac{n}{n-1}}
 \ge c_{B\!Z}\big(\text{Vol }M_0\big) > 0,
\]
where we also used the isoperimetric inequality and the fact that volume is monotonic
increasing under this flow.  

Following a similar line of reasoning gives rise to several other `conservation' type flows.  For
example, with $h=h_K$ we calculate
\begin{align*}
\rD{}{t}\int_M H d\mu
   &= \int_M \big[(H^2-\vn{A}^2)(\Delta H + h_K) - \Delta^2 H\big] d\mu\\
   &= \int_M K(\Delta H) d\mu + h_K\int_M K d\mu = 0,
\end{align*}
where $\vn{A}^2$ denotes the squared norm of the second fundamental form of $M_t$.  Thus the
generalised mixed volume $\int_M H d\mu$ is always preserved under \eqref{CSD} flow with $h=h_K$.
In this case $\int_M K d\mu$ is the denominator of $h_K$, which is constant under the flow, and so
similarly to $\hH$ the constraint function $h_K$ is always defined.  One expects that global
analysis of flows such as this, which preserve a geometrically interesting quantity or keep it
monotone in time, would
lead to new geometric inequalities.  At the very least we would expect to obtain new proofs of
classical geometric inequalities, such as the isoperimetric inequality.  This is in direct analogy
with the work of Huisken \cite{huisken1986vpm}
and the first author \cite{J1,J2,J3} for example.


A first step in any program of analysis for these flows is a short time existence theorem.
The first appearance of such a theorem in the context of geometric heat flows in the literature is due to
Huisken-Polden \cite{huisken1999gee,polden:cas}.  
%
While the idea of proof there is clear, the usage of the linearisation is not.
This was later clarified in a much more restricted case by Sharples
\cite{sharples2004laq}, who considers only second order flows, but claims the techniques are
applicable also to the higher order case \cite{sharplesprivate}.
Independently, Escher, Mayer and Simonett \cite{escher98surface} apply theorems credited to Amann to
conclude short time existence for \eqref{SD} flow.  Unfortunately the quoted references 
are not readily available.

Despite this confusion, there is a much more standard approach to the problem of short time
existence for our flows \eqref{CSD} pointed out by Kuwert \cite{kuwertprivate}.
One may adapt the existence and uniqueness theory for higher order quasilinear parabolic partial
differential equations in $\R^n$.  This is easily accomplished by writing the problem as a graph,
and then we must consider a degenerate quasilinear fourth order parabolic partial differential
equation.

Depending on the constraint function, short time existence for this equation with $f_0$ at least
$C^4(M_0)$ follows from (for example) the linear estimates found in Eidel'man and Zhitarashu
\cite{eidelman1998pbv}, Solonnikov \cite{solonnikov1965bvp}, or an extension of those in Friedman
\cite{friedman1964pde}, combined with a fixed point argument.  Uniqueness can be obtained by a
method similar to that found in Li \cite{shuanhuli}, which is originally due to unpublished notes of
Amann.  The relevant theorem
is also stated in Amann \cite{amann1993nla}.  

Now, depending on the constraint function $h$ there are two possible approaches: if we have a known
function of $t$, such as $\frac{1}{1+t}$, $\sin t$, and so on, then 
one must show that $\partial_th(0)$ is bounded.  Otherwise, if we have
a constraint function consisting of integrals as above with $h_K$ and $h_H$, we use the initial
smoothness of the immersion $f_0$ to guarantee estimates for $h$.  For example, in the case where
$h=h_H$, there are up to seven derivatives of the immersion in $\partial_th$, and so if $f_0 \in
C^7(M_0)$, we will have a short time existence theorem.

Therefore one can see that the regularity of $f_0$ required to obtain short time existence is at
least $C^4(M_0)$, and if $h$ consists of integrals of curvature then the required regularity could
be quite high, depending on $h$.  This is what we mean below when we say `smooth enough'. 

\begin{thm}[Short time existence] \label{thmSTE}
For any smooth enough initial immersion $f_0:M^n\rightarrow\R^{n+1}$ and constraint function
$h:I\rightarrow\R$ with $I$ an interval containing $0$ and $h\in C^1(I)$, there exists
a unique nonextendable smooth solution $f:M\times[0,T)\rightarrow\R^{n+1}$ to \eqref{CSD} with
$f(\cdot,0) = f_0$, where $0 < T \le \infty$.  
\end{thm}

For $T\in(0,\infty]$ a solution $f:M\times[0,T)\rightarrow\R^{n+1}$ to \eqref{CSD} is nonextendable
if $T\ne\infty$ and there is no $\delta>0$ such that
$\tilde{f}:M\times[0,T+\delta)\rightarrow\R^{n+1}$ is also a solution to \eqref{CSD} with
$\tilde{f}(\cdot,0) = f_0$.  For the interested reader, a more detailed discussion of this theorem
can be found in \cite{mythesis}.


Motivated by the observation that \eqref{SD} flow can also be derived by considering the
$H^{-1}$-gradient flow for the area functional (see Fife \cite{fife2000mps}), and the recent work of
Kuwert \& Sch\"atzle \cite{kuwert2001wfs,kuwert2002gfw} on the gradient flow for the Willmore
functional, we present the following theorem.

\begin{thm}[Lifespan Theorem]
\label{t_lifespan}
Suppose $n\in\{2,3\}$ and let $f:M^n\times[0,T)\rightarrow\R^{n+1}$ be a closed immersion with
$C^\infty$ initial data evolving by 
\begin{equation}
\tag{CSD}
\pD{}{t}f = (\Delta H + h)\nu.
\end{equation}
Then there are constants $\rho>0$, $\epsilon_0>0$, and $c<\infty$ such that if $h : I \supset
[0,\frac{1}{c}\rho^4] \rightarrow \R$ is a function satisfying 
\begin{equation}
\label{A1}
\tag{A1}
\vn{h}_{\infty,I} < \infty,
\end{equation}
and $\rho$ is chosen with
\begin{equation}
\label{eq1}
\int_{f^{-1}(B_\rho(x))} \vn{A}^n d\mu\Big|_{t=0} = \epsilon(x) \le \epsilon_0
\qquad
\text{ for any $x\in\R^{n+1}$},
\end{equation}
then for $n=2$ the maximal time $T$ of smooth existence for the flow \eqref{CSD} with initial data $f_0 =
f(\cdot,0)$ satisfies
\begin{equation}
\label{eq2}
T \ge \frac{1}{c}\rho^4,
\end{equation}
and we have the estimate
\begin{equation}
\label{eq3}
\int_{f^{-1}(B_\rho(x))} \vn{A}^n d\mu \le c\epsilon(x)
\qquad\text{ for }\qquad
0\le t \le \frac{1}{c}\rho^4.
\end{equation}
For $n=3$, the conclusions \eqref{eq2}, \eqref{eq3} hold under the additional assumption that
there exists an absolute constant $C_{A\!B}\in(0,\infty)$ such that
\begin{equation}
\label{AB}
\tag{AB}
|M_t| \le C_{A\!B},
\qquad\text{ for }\qquad 0 \le t \le \frac{1}{c}\rho^4.
\end{equation}
\end{thm}

The restriction on the dimension of the evolving immersion is due to both the exponent in the
Michael-Simon Sobolev inequality, and the scaling of the total squared curvature functional.
For flows where the evolution of the surface area is bounded (such as \eqref{SD}
and \eqref{CSD} with $h=\hH$) we have removed the latter restriction by considering
\eqref{eq1},
which is a natural generalisation of (1.4) in \cite{kuwert2002gfw}.  The size of $\epsilon_0$ is
determined indirectly by the bound on surface area for the flow in question.
As to the exponent in the Michael-Simon Sobolev inequality, the interplay between the evolution
equations and our techniques using integral estimates forces $n < 4$; see Section 5 for a discussion
of this issue.  

At first glance, the choice in \eqref{eq1} may appear somewhat restrictive, since $\epsilon_0$ (the
size of which is dictated by estimates to come) may be very small.  However, it is clear that if the
initial surface $M_0$ is of finite total curvature (that is, $\int_{M} \vn{A}^n d\mu\big|_{t=0} <
\infty$), then there will exist a positive $\rho = \rho(\epsilon_0, M_0)$ such that \eqref{eq1} is
satisfied.  Therefore, in terms of allowable initial surfaces $M_0$, we are only excluding those for
which the total curvature is infinite.

The assumption \eqref{A1} also appears restrictive.  For an a priori known function of time, it is
appropriate, but for our given examples ($h=\hH, h_K$) it is not clear that \eqref{A1} is
satisfied.  We will show in Section 3 that constraint functions similar to $\hH$ admit an a
priori bound, while constraint functions of a form similar to $h_K$ remain just beyond our current
techniques.  It is in this sense which the two examples serve to differentiate between those
constraint functions which are relatively easy to handle, and those which just present difficulty.
The inequality
\[
\sup_{x,y\in f(M)}|x-y| = \text{extrinsic diameter} = d_{ext}
                        \le c_T(n)\int_{M} |H|^{n-1}d\mu
\]
due to Topping \cite{topping2008}
 will also play a major role,
allowing us to prescribe a class of constraint functions which admit a `localisation' procedure.
The extra assumptions required will be a growth condition, and a geometric condition: either bounded
surface area or bounded total mean curvature.


In a more global sense, we present the lifespan theorem with a perspective toward further analysis
of the \eqref{CSD} flows.  In particular, as the statement depends on the concentration of the
curvature of the initial surface, the result is particularly relevant to the analysis of asymptotic
behaviour in the following respect.  When considering a blowup of a singularity formed at some time
$T<\infty$ of the \eqref{CSD} flow, we wish to have that some amount of the curvature concentrates
in space.  From the theorem, if $\rho(t)$ denotes the largest radius such that \eqref{eq1} holds at
time $t$, then $\rho(t) \le \sqrt[4]{c(T-t)}$ and so at least $\epsilon_0$ of the curvature
concentrates in a ball $f^{-1}(B_{\rho(T)}(x))$.  That is,
\[
\lim_{t\rightarrow T} \int_{f^{-1}(B_{\rho(t)}(x))} \vn{A}^n d\mu \ge \epsilon_0,
\]
where $x = x(t)$ is understood to be the centre of a ball where the integral above is maximised.

As already mentioned, our motivation for the extension of \eqref{SD} to the more general class of
flows \eqref{CSD} is essentially mathematical.  However there does already exist a large body of
work on \eqref{SD} flow
itself, and study of \eqref{SD} alone is well motivated.  First proposed by the physicist Mullins
\cite{mullins1957ttg} in 1957 (two years before he proposed the mean curvature flow), it was
originally designed to model the formation of tiny thermal grooves in phase interfaces where the
contribution due to evaporation-condensation was insignificant.  Some time later, Davi, Gurtin, Cahn
and Taylor \cite{cahn1994sms,davi1990mpi} proposed many other physical models which give rise to
the surface diffusion flow.  These all exhibit a reduction of free surface energy and conservation
of volume; an essential characteristic of \eqref{SD} flow.  There are also other motivations for the
study of \eqref{SD}.  For example, two years later Cahn, Elliot and Novick-Cohen \cite{cahn1996che}
proved that \eqref{SD} is the singular limit of the Cahn-Hilliard equation with a concentration
dependent mobility.  Among other applications, this arises in the modeling of isothermal separation
of compound materials.

Analysis of the surface diffusion flow began slowly, with the first works appearing in the early
80s.  Baras, Duchon and Robert \cite{baras1984edi} showed the global existence of weak solutions for
two dimensional strip-like domains in 1984.  Later, in 1997 Elliot and Garcke \cite{elliott1997erd}
analysed \eqref{SD} flow of curves, and obtained local existence and regularity for $C^4$-initial
curves, and global existence for small perturbations of circles.  Significantly, Ito
\cite{Ito1999sdf} showed in 1998 that convexity will not be preserved under \eqref{SD}, even for
smooth, rotationally symmetric, closed, compact, strictly convex initial hypersurfaces.  In contrast
with the case for second order flows such as mean curvature flow, this behaviour appears
pathological.  Escher, Mayer and Simonett \cite{escher98surface} gave several numerical schemes for
modeling \eqref{SD} flow, and have also given the only two known numerical examples
\cite{mayer2001nss} of the development of a singularity: a tubular spiral and thin-necked dumbbell.
They also provide an example of an immersion which will self-intersect under the flow, a figure eight
knot.  In 2001, Simonett \cite{simonett2001wfn} used centre manifold techniques to show that for
initial data $C^{2,\alpha}$-close to a sphere, both the surface diffusion and Willmore flows
(Willmore flow in one codimension is $\partial_t f = \Delta H + \vn{A^o}^2H$, where $A^o = A -
\text{trace}_g\ A$) exist for all time and converge asymptotically to a sphere.

There have been many 
important works on fourth order flows of a slightly different character,
from Willmore flow of surfaces
 to Calabi flow, a fourth order flow of metrics.  Significant contributions to the analysis of these
flows by the authors Kuwert, Sch\"atzle, Polden, Huisken, Mantegazza and Chru\'sciel
\cite{chrusciel1991sge,kuwert2001wfs,kuwert2002gfw,mantegazza2002sge,polden:cas} are particularly
relevant, as the methods employed there are similar to ours here.

In our proof, we exploit the fact that for an $n$-dimensional immersion the integral
\[
\int_M \vn{A}^n d\mu
\]
is scale invariant.  The technique used by Struwe \cite{struwe1985ehm} 
is then
relevant, although as with all higher order flows the major difficulty is in overcoming the lack of
powerful techniques unique to the second order case.  In particular, we are without the maximum
principle, and this implies that the geometry of the surface could deteriorate, as in
\cite{Ito1999sdf}.  Therefore we are forced to use integral estimates to derive derivative curvature
bounds under a condition similar to \eqref{eq1}, and in calculating these estimates it is crucial to
only use inequalities which involve universal constants.  Interpolation inequalities similar in
nature to those used by Ladyzhenskaya, Ural'tseva and Solonnikov \cite{ladyzhenskaya1968laq} and
Hamilton \cite{RH}, and the Sobolev inequality of Michael-Simon \cite{michael1973sam}, are
invaluable in this regard.

The structure of this paper is as follows.  To apply the argument used by Struwe, we must prove two
key local integral estimates.  In Section 2 we collect various fundamental formulae from
differential geometry, set our notation, and state some basic results.  The goal of Section 3 is to
show that the a priori bound \eqref{A1} is satisfied by a class of constraint functions, and to
detail the localisation procedure required to use the 
global constraint function in
local integral estimates.  Section 4 is concerned with estimating the evolution of local integrals
of derivatives of curvature.  Section 5 combines these estimates with Sobolev inequalities,
interpolation inequalities, and the results of Section 3 to conclude the two required key integral
estimates.  With these in hand, we adapt the argument of Struwe in Section 6 to prove the lifespan
theorem.  Section 7 contains some remarks on lifespan theorems for flows similar to \eqref{CSD}.


We note that a theorem similar to Theorem \ref{t_lifespan} was proposed in \cite{LU}, applying only
to the flow \eqref{SD}.  Our work includes a proof of this result, when $n=2$ and $n=3$.




\section{Notation and preliminary results.}

In this section we will collect various general formulae from differential geometry which we will
need when performing the later analysis.  We will adopt similar notation to Hamilton \cite{RH} and Huisken
\cite{huisken1984fmc}. We have as our principal object of study a smooth immersion
$f:M^n\rightarrow\R^{n+1}$ of an orientable compact hypersurface $M$, and induced metric tensor with
components
\[
g_{ij} = \eIP{\pD{}{x_i}f}{\pD{}{x_j}f},
\]
so that the pair $(M,g)$ is a Riemannian manifold.  In the above equation $\eIP{\cdot}{\cdot}$
denotes the regular Euclidean inner product, and $\pD{}{x_i}$ is the derivative in the direction of
the $i$-th basis vector of the ambient space, which in our case is the regular Euclidean partial
derivative.  When convenient we frequently use the abbreviation $\partial_i = \pD{}{x_i}$.

The Riemannian metric induces an inner product structure on all tensors,
which we define as the trace over pairs of indices with the metric:
\[
\IP{T^{i}_{jk}}{S^i_{jk}} = g_{is}g^{jr}g^{ku}T^i_{jk}S^s_{ru},\qquad \vn{T}^2 = \IP{T}{T},
\]
where repeated indices are summed over from $1$ to $n$.
The mean curvature $H$ is defined by
\[
H = g^{ij}A_{ij} = A_i^i,
\]
where the components $A_{ij}$ of the second fundamental form $A$ are given by
\begin{equation}
A_{ij} = -\eIP{\pD{{}^2}{x_i\partial x_j}f}{\nu}
       = \eIP{\pD{}{x_j}f}{\pD{}{x_i}\nu},
\label{PREPsff}
\end{equation}
where $\nu$ is the outer unit normal vector field on $M$.

The Christoffel symbols of the induced connection are determined by the metric,
\begin{equation}
\label{C3Echristoffelmetric}
\Gamma_{ij}^k = \frac{1}{2}g^{kl}
                \left(\pD{}{x_i}g_{jl} + \pD{}{x_j}g_{il} - \pD{}{x_l}g_{ij}\right),
\end{equation}
so that then the covariant derivative on $M$ of a vector $X$ and of a covector $Y$ is
\begin{align*}
\nabla_jX^i &= \pD{}{x_j}X^i + \Gamma^i_{jk}X^k\text{, and}\\
\nabla_jY_i &= \pD{}{x_j}Y_i - \Gamma^k_{ij}Y_k
\end{align*}
respectively.

From the expression \eqref{PREPsff} and the smoothness of $f$ we can see that the second fundamental form is
symmetric; less obvious but equally important is the symmetry of the first covariant derivatives of
$A$, 
\[ \nabla_iA_{jk} = \nabla_jA_{ik} = \nabla_kA_{ij}, \]
commonly referred to as the Codazzi equations.

The fundamental relations between components of the Riemann curvature tensor $R_{ijkl}$, the Ricci
tensor $R_{ij}$ and scalar curvature $R$ are given by Gauss' equation
\begin{align*}
R_{ijkl} &= A_{ik}A_{jl} - A_{il}A_{jk},\intertext{with contractions}
g^{jl}R_{ijkl}
   = R_{ik} 
  &= HA_{ik} - A_i^jA_j^k\text{, and}\\
g^{ik}R_{ik}
   = R
  &= H^2 - \vn{A}^2.
\end{align*}
We will need to interchange covariant derivatives; for vectors $X$ and covectors $Y$ we obtain
\begin{align*}
\nabla_{ij}X^h - \nabla_{ji}X^h &= R^h_{ijk}X^k = (A_{lj}A_{ik}-A_{lk}A_{ij})g^{hl}X^k,\\
\nabla_{ij}Y_k - \nabla_{ji}Y_k &= R_{ijkl}g^{lm}Y_m = (A_{lj}A_{ik}-A_{il}A_{jk})g^{lm}Y_m,
\end{align*}
where $\nabla_{i_1\ldots i_n} = \nabla_{i_1} \cdots \nabla_{i_n}$.  Further we define $\nabla_{(n)}T$
to be the tensor with components $\nabla_{i_1\ldots i_n}T_{j_1\ldots}^{k_1\ldots}$.  We also use for
tensors $T$ and $S$ the notation $T*S$ (as in Hamilton \cite{RH}) to denote a linear combination of
new tensors, each formed by contracting pairs of indices from $T$ and $S$ by the metric $g$ with
multiplication by a universal constant.  The resultant tensor will have the same type as the other
quantities in the equation it appears.  Keeping these in mind we also denote polynomials in the
iterated covariant derivatives of these terms by \[ P_j^i(T) = \sum_{k_1+\ldots+k_j = i}
c\nabla_{(k_1)}T*\cdots*\nabla_{(k_j)}T, \] where the constant $c\in\R$ is absolute and may vary
from one term in the summation to another.  As is common for the $*$-notation, we slightly abuse
this constant when certain subterms do not appear in our $P$-style terms. For example \begin{align*}
\vn{\nabla A}^2 &= \IP{\nabla A}{\nabla A}\\ &= 1\cdot\left(\nabla_{(1)}A*\nabla_{(1)}A\right) +
0\cdot\left(A*\nabla_{(2)}A\right)\\ &= P_2^2(A).  \end{align*} This will occur throughout the paper
without further comment.

The Laplacian we will use is the Laplace-Beltrami operator on $M$, with the components of $\Delta T$
given by
\[
\Delta T^i_{jk} = g^{pq}\nabla_{pq}T^i_{jk} = \nabla^p\nabla_pT^i_{jk}.
\]
Using the Codazzi equation with the interchange of covariant derivative formula given above, we
obtain Simons' identity:
\begin{align}
\Delta A_{ij} &= \nabla_{ij}H + HA_{il}g^{lm}A_{mj} - \vn{A}^2A_{ij}\notag\\
              &= \nabla_{ij}H + HA_{i}^lA_{lj} - \vn{A}^2A_{ij},\notag\\
\intertext{or in $*$-notation}
\Delta A &= \nabla_{(2)}H + A*A*A.              \tag{SI}\label{SimonsIdentity}
\end{align}
In the coming sections we will be concerned with calculating the evolution of the iterated covariant
derivatives of curvature quantities.  The following less precise interchange of covariant
derivatives formula (derived from the fundamental equations above) will be useful to keep in mind:
\[
\nabla_{ij}T = \nabla_{ji}T + P_2^0(A)*T.
\]
In most of our integral estimates (especially those in sections 4 and 5), we will be including a
function $\gamma:M\rightarrow\R$ in the integrand.  Eventually, this will be specialised to a smooth
cutoff function between concentric geodesic balls on $M$.  For these estimates however, we will only
assume that $\gamma = \tilde{\gamma}\circ f$, where 
\begin{equation*}
0\le\tilde{\gamma}\le 1,\qquad\text{ and }\qquad 
\vn{\tilde{\gamma}}_{C^2(\R^{n+1})} \le c_{\tilde{\gamma}} < \infty.
\end{equation*}
Using the chain rule, this implies $D\gamma = (D\tilde{\gamma}\circ f)Df$ and then
$D^2\gamma = (D^2\tilde{\gamma}\circ f)(Df,Df) + (D\tilde{\gamma}\circ f)D^2f(\cdot,\cdot)$.
Using the expression
\eqref{C3Echristoffelmetric} for the Christoffel symbols to convert the computations above to
covariant derivatives, and the Weingarten relations
\[
\partial_i\nu = A_{i}^j\partial_jf,\qquad \partial_i\partial_jf=-A_{ij}\nu, 
\]
to convert the derivatives of $\nu$ to factors
of the second fundamental form with the basis vectors $\partial_if$, we obtain the estimates
\begin{equation}
\tag{$\gamma$}
\label{e:gamma}
\vn{\nabla\gamma} \le c_{\gamma1},\qquad\text{ and }\qquad
\vn{\nabla_{(2)}\gamma} \le c_{\gamma2}(1+\vn{A}).
\end{equation}
For a given $\rho>0$, we also define the functions $\epsilon,
\DP:\R^{n+1}\times[0,T^*]\rightarrow\R$ as
\[
\epsilon(x) = \int_{f^{-1}(B_\rho(x))}\vn{A}^2d\mu,
\quad\text{ and }\quad
\DP(x) = \int_{f^{-1}(B_\rho(x))}\vn{A}^pd\mu.
\]
At times we will instead consider the set $[\gamma>0] = \{q\in M:\gamma(q)>0\}$ as the domain of
the integrals in $\epsilon(x)$ and $\DP(x)$.  

\section{A priori estimates for the constraint function.}

Our constraint functions are by their nature global notions (being functions of time only).  This is
a distinct advantage in some areas of the analysis: evolution equations first order in time and of
any order in space involve at most a linear factor of $h$.

When one wishes to prove local integral estimates however, the global nature of $h$ becomes an
issue.  We are faced with situations such as
\begin{align}
\rD{}{t}\int_{f^{-1}(B_\rho(x))}\vn{A}^2d\mu &+ \int_{f^{-1}(B_\rho(x))}\vn{\nabla_{(2)}A}^2d\mu
\notag
\\*
\label{hsec1}
 &\le h\int_{f^{-1}(B_{2\rho}(x))}\big(\vn{A}^3 + \vn{A}^2\big)d\mu + \text{``good terms''},
\end{align}
armed with a local smallness of curvature assumption
\[
\sup_{x\in\R^{n+1}\atop t\in[0,T^*]} \epsilon(x) \le \epsilon_0,\quad
\text{ or }
\quad\sup_{x\in\R^{n+1}\atop t\in[0,T^*]}\DP(x) \le \delta_0,
\]
and tasked with absorbing the term involving $h$, a global term, into
\[
\int_{f^{-1}(B_\rho(x))}\vn{\nabla_{(2)}A}^2d\mu,
\]
a local integral.  Assume for the sake of example that $h = \int_M k(\mathcal{W})d\mu$, where
$\mathcal{W}$ is the Weingarten map, and $h$ obeys an estimate
\[
h \le C_{A\!B\!S}\int_M \vn{A}^2d\mu \int_M \vn{\nabla_{(2)}A}^2d\mu,
\]
where $C_{A\!B\!S}$ is an absolute constant.  Then as a first attempt to `localise' the integrals on
the right one might estimate them by
\begin{align*}
\int_M \vn{A}^2d\mu& \int_M \vn{\nabla_{(2)}A}^2d\mu
\\
&\le
    c^2_\rho(t)\sup_{x\in\R^{n+1}}\int_{f^{-1}(B_\rho(x))} \vn{A}^2d\mu
               \sup_{x\in\R^{n+1}}\int_{f^{-1}(B_\rho(x))}\vn{\nabla_{(2)}A}^2d\mu
\\
&\le 
    c^2_\rho(t)\epsilon_0
\int_{f^{-1}(B_\rho(x_1))}\vn{\nabla_{(2)}A}^2d\mu,
\end{align*}
where $c_\rho(t)$ is the number of extrinsic balls of radius $\rho$ required to cover $f(M_t)$ and
$x_1\in\R^{n+1}$ is a point where the second supremum is attained.  The goal of course is to now
bound $c^2_\rho(t)\epsilon_0$ by $\frac{1}{2C_{A\!B\!S}}$ (for example), and absorb the entire term
on the left in \eqref{hsec1}.  Unfortunately, this will in general not be possible.  To attain a
smaller $\epsilon_0$, one must drive $\rho$ to zero, but this will in turn drive $c_\rho$ to
$\infty$.  Further, the scaling is unfavourable, making it difficult to know a priori if any
admissible $\rho > 0$ exists.  Finally, $c_\rho$ is a function of time, and without a uniform bound
we have little hope of absorbing the constraint function into a local integral.

With some minor modifications to the above idea, and assumptions on the flow, these problems can be
overcome and the argument carries through.  Our main result for this section is the following.

\begin{thm}
\label{thmapriorih}
Let $f:M^n\times[0,T^*]\rightarrow \R^{n+1}$ be a \eqref{CSD} flow with constraint function $h$
satisfying for some $j,k,l\in\N_0$
\begin{equation}
h \le \int_M P_j^2(A) + P_k^1(A) + P_l^0(A) d\mu
\label{ERRATAgrowthoncosntraintcondition}
\tag{A2}
\end{equation}
where for $m = \max\{2k-2,2j-k,l,n^2+n-2\}$ 
\[
\sup_{x\in\R^{n+1}}\delta^{(m)}(x) \le \delta^{(m)}_0 < \infty,
\]
and for a finite absolute constant $C_{A\!B}$
\begin{equation}
\tag{AB}
|M_t| \le C_{A\!B};
\end{equation}
on $[0,T^*]$.
 
Then for any $\rho > 0$, $x\in\R^{n+1}$, $t\in[0,T^*]$ there exists an $x_1\in\R^{n+1}$ such that
for any $\theta > 0$,
\[
h\int_{f^{-1}(B_{2\rho}(x))} \big(\vn{A}^4 + \vn{A}^2\big)d\mu
 \le \theta\int_{f^{-1}(B_\rho(x_1))} \vn{\nabla_{(2)}A}^2d\mu 
+ C_{U\!G\!L\!Y},
\]
if $j,k\ne0$, and otherwise
\[
h\int_{f^{-1}(B_{2\rho}(x))} \big(\vn{A}^4 + \vn{A}^2\big)d\mu
 \le C_{U\!G\!L\!Y},
\]
where $C_{U\!G\!L\!Y} = C_{U\!G\!L\!Y}(\theta,\delta^{(m)}_0,C_{A\!B},\rho,j,k,l,n)$.
\end{thm}

Before we begin the proof we would like to show that $\hH$ satisfies the assumptions of
the theorem.  By viewing mean curvature as the variation of area, 
Burago-Zalgaller \cite{burago1988gi} prove
the estimate
\begin{equation}
|M| \le c\bigg(\int_M |H| d\mu\bigg)^\frac{n}{n-1}
\label{inBZ}
\end{equation}
for a constant $c$ depending only on $n$.  Using now the isoperimetric inequality we conclude
\[
\frac{1}{\int_M|H|d\mu} \le c{|M|}^\frac{1-n}{n} \le c(\text{Vol }M)^{-1}
\le c\text{Vol }M_0.
\]
Therefore we may estimate
\[
\hH(t)
 = \frac{\int_M \vn{\nabla H}^2 d\mu}{\int_M |H| d\mu}
 \le c(M_0)\int_M P_2^2(A) d\mu.
\]
Thus for any dimension $n$ we take $m = (n-1)(n+2)$. Also, \eqref{AB} is
satisfied with
\[
C_{A\!B} = |M_0|.
\]

Driving Theorem \ref{thmapriorih} is the following estimate due to Topping \cite{topping2008}.

\begin{thm}
\label{topping}
Let $M^n$ be a compact connected $n$-dimensional submanifold of $\R^{n+1}$.  Then its extrinsic
diameter and its mean curvature $H$ are related by
\[
  d_{ext} \le c_T(n)\int_M |H|^{n-1}d\mu.
\]
\end{thm}

Topping shows that in particular we may take $c_T(2) = \frac{32}{\pi}$.  We refer the reader to the
references in \cite{topping2008} and \cite{simon1993esm} for a history of this inequality and others
similar to it.  

We first obtain an estimate for $c_\rho(t)$.

\begin{lem}
Let $f:M^n\times[0,T^*]\rightarrow\R^{n+1}$ be a \eqref{CSD} flow satisfying \eqref{AB}.
Then for any $\rho$ such that $0 < \rho \le \frac{d_{ext}\sqrt{n+1}}{2}$ there exists an
$x_2\in\R^{n+1}$ where the following estimate holds:
\[
  c_\rho(t) \le c(C_{A\!B}, \rho,
n)\bigg(\int_{f^{-1}(B_\rho(x_2))}\vn{A}^{(n-1)(n+2)}d\mu\bigg)^{n+1}.
\]
\end{lem}

\begin{rmk} If $\rho > \frac{d_{ext}\sqrt{n+1}}{2}$ then $c_\rho(t) = 1$.  We will always assume
from now on that $0 < \rho \le \frac{d_{ext}\sqrt{n+1}}{2}$.
\end{rmk}

\begin{proof}  We simply apply a covering argument, Theorem \ref{topping}, and then the H\"older
inequality.  Since we can cover $M_t$ by an $(n+1)$-cube with side length ${d_{ext}}$ and a ball of
radius $\rho$ encloses an $(n+1)$-cube with side length $\frac{2\rho}{\sqrt{n+1}}$,
\begin{align*}
c_\rho(t)
  &\le \bigg( \frac{d_{ext}\sqrt{n+1}}{2\rho} \bigg)^{n+1}
\\
  &\le \bigg( \frac{c_T(n)\sqrt{n+1}}{2\rho} \bigg)^{n+1}  \bigg(\int_M |H|^{n-1}d\mu\bigg)^{n+1}
\\
  &\le \bigg( \frac{c_T(n)\sqrt{n+1}}{2\rho} \bigg)^{n+1}  |M_t|^{\frac{(n+1)^2}{n+2}}
       \bigg(\int_M |H|^{(n-1)(n+2)}d\mu\bigg)^{\frac{n+1}{n+2}}
\\
  &\le \bigg( \frac{c_T(n)\sqrt{n+1}}{2\rho} \bigg)^{n+1}  |M_t|^{\frac{(n+1)^2}{n+2}}
       \bigg(\sup_{x\in\R^{n+1}}c_\rho(t)\int_{f^{-1}(B_\rho(x))}
|H|^{(n-1)(n+2)}d\mu\bigg)^{\frac{n+1}{n+2}},
\intertext{so}
c_\rho(t)
  &\le \bigg( \frac{c_T(n)\sqrt{n+1}}{2\rho} \bigg)^{(n+1)(n+2)}  C_{A\!B}^{(n+1)^2}
       \bigg(\int_{f^{-1}(B_\rho(x_2))} \vn{A}^{(n-1)(n+2)}d\mu\bigg)^{n+1},
\end{align*}
where $x_2$ is a point in $\R^{n+1}$ such that
\[
\int_{f^{-1}(B_\rho(x_2))} \vn{A}^{(n-1)(n+2)}d\mu
 = \sup_{x\in\R^{n+1}} \int_{f^{-1}(B_\rho(x))} \vn{A}^{(n-1)(n+2)}d\mu.
\]
\end{proof}
\begin{rmk}
Since we can take $c_T(2) = \frac{32}{\pi}$, the conclusion in the theorem above for a \eqref{CSD}
flow with $h=\hH$ and $n=2$ is
\begin{equation*}
c_\rho(t)
  \le \bigg(\frac{32\sqrt{3}}{2\pi\rho}\bigg)^{12}|M_0|^9
       \bigg(\int_{f^{-1}(B_\rho(x_2))} \vn{A}^{4}d\mu\bigg)^{3}.
\end{equation*}
\end{rmk}
We now use the above to estimate $h$.

\begin{lem} Let $\theta>0$ be a fixed positive number and $f:M\times[0,T^*]\rightarrow\R^{n+1}$ a
\eqref{CSD} flow satisfying the assumptions of Theorem \ref{thmapriorih}.  Then for any $\rho > 0$
there exists a point $x_1\in\R^{n+1}$ such that the constraint function $h$ satisfies the following
estimate:
\[
h \le \theta\int_{f^{-1}(B_\rho(x_1))} \vn{\nabla_{(2)}A}^2d\mu
 + c(\theta,\rho,n,j,k,l,C_{A\!B},\delta_0^{(m)})\delta_0^{(m)}
\]
for $j,k,l\ne0$, and
\[
h  \le 
  c(\rho,n,C_{A\!B})\big(\delta_0^{(m)}\big)^{n+1}
\]
for $j=k=l=0$.
\label{lemalltheworkapriorih}
\end{lem}
\begin{proof}
Recall that 
\[
\sup_{x\in\R^{n+1}}\delta^{(m)}(x) \le \delta^{(m)}_0 < \infty,
\]
where $m = \max\{2j-2,2k-j,l,n^2+n-2,4\}$.

We will first prove the estimate assuming that $j\ge\max\{2,2k+1\}$:
\begin{align*}
h &\le \int_M P_j^2(A) + P_k^1(A) + P_l^0(A) d\mu
\\
&\le c\sup_{x\in\R^{n+1}}\Big(c_\rho\int_{f^{-1}(B_\rho(x))}
                               \vn{\nabla_{(2)}A}\cdot\vn{A}^{j-1}d\mu\Big)
 + \int_M\vn{\nabla A}^2\vn{A}^{j-2} d\mu
\\
&\qquad + c\int_M \vn{\nabla A}\cdot\vn{A}^{k-1} d\mu
 + c\int_M \vn{A}^l d\mu
\\
&\le \frac{\theta}{2}\int_{f^{-1}(B_\rho(x_1))} \vn{\nabla_{(2)}A}^2d\mu
 + c_\rho^2\frac{c}{2\theta}\int_{f^{-1}(B_\rho(x_1))} \vn{A}^{2j-2}d\mu
\\*
&\qquad
 + c\int_M \vn{\nabla A}^2\vn{A}^{j-2} d\mu
 + c\int_M \vn{A}^{2k-j} + \vn{A}^l d\mu
\\
&\le \frac{\theta}{2}\int_{f^{-1}(B_\rho(x_1))} \vn{\nabla_{(2)}A}^2d\mu
 + c\Big(1+(j-2)\frac{4-j}{j-3}\Big)\int_M \IP{A}{\Delta A}\vn{A}^{j-2} d\mu
\\*
&\qquad
 + c_\rho^2\frac{c}{2\theta}\int_{f^{-1}(B_\rho(x_1))} \vn{A}^{2j-2}d\mu
 + c(\theta,j,k,l)\int_M \vn{A}^{2k-j} + \vn{A}^l d\mu
\\
&\le \theta\int_{f^{-1}(B_\rho(x_1))} \vn{\nabla_{(2)}A}^2d\mu
 + c_\rho^2c(\theta,j,k)\int_{f^{-1}(B_\rho(x_1))} \vn{A}^{2j-2} d\mu
\\*
&\qquad
 + c(\theta,j,k,l)\int_M \vn{A}^{2k-j} + \vn{A}^l d\mu
\\
&\le \theta\int_{f^{-1}(B_\rho(x_1))} \vn{\nabla_{(2)}A}^2d\mu
 + c_\rho^2c(\theta,j,k,C_{A\!B})\bigg(\int_{f^{-1}(B_\rho(x_1))} \vn{A}^{m} d\mu\bigg)^{\frac{2j-2}{m}}
\\*
&\qquad
 + c(\theta,j,k,l,C_{A\!B})\bigg(
              \sup_{x\in\R^{n+1}}c_\rho\int_{f^{-1}(B_\rho(x))} \vn{A}^{m} d\mu
                           \bigg)^{\frac{2k-j+l}{m}}
\\
&\le \theta\int_{f^{-1}(B_\rho(x_1))} \vn{\nabla_{(2)}A}^2d\mu
 + c_\rho^{\frac{2m+2k-j+l}{m}}c(\theta,j,k,l,C_{A\!B})\big(\delta_0^{(m)}\big)^{\frac{j-2+2k+l}{m}}
\\
&\le \theta\int_{f^{-1}(B_\rho(x_1))} \vn{\nabla_{(2)}A}^2d\mu
 + c(\theta,\rho,n,j,k,l,C_{A\!B})\big(\delta_0^{(m)}\big)^{\frac{(n+1)(2m+2k-j+l)+j-2+2k+l}{m}}.
\end{align*}
The estimate is easier to prove in the subcases excluded above.  When $j=1$ we instead split the
first integral by
\begin{align*}
\int_M P_j^2(A)d\mu
&\le c\sup_{x\in\R^{n+1}}c_\rho\int_{f^{-1}(B_\rho(x))}
                               \vn{\nabla_{(2)}A}d\mu
\\
&\le \frac{\theta}{2}\int_{f^{-1}(B_\rho(x_3))}
                               \vn{\nabla_{(2)}A}^2d\mu
  + c_\rho^2\frac{2}{\theta}\int_{f^{-1}(B_\rho(x_3))}
                               1\ d\mu
\\
&\le \frac{\theta}{2}\int_{f^{-1}(B_\rho(x_3))}
                               \vn{\nabla_{(2)}A}^2d\mu
 + c(\theta,\rho,n,C_{A\!B})\big(\delta_0^{(m)}\big)^{2n+2}.
\end{align*}
When $j<2k+1$ we instead estimate the second integral by
\begin{align*}
\int_M P_k^2(A)d\mu
&\le c\int_M \vn{\nabla A}\cdot\vn{A}^{k-1}d\mu
\\
&\le c\int_M \vn{\nabla A}^2d\mu
   + c\int_M \vn{A}^{2k-2}d\mu
\\
&\le c\int_M \vn{\nabla_{(2)} A}\cdot\vn{A}d\mu
   + c\int_M \vn{A}^{2k-2}d\mu
\\
&\le \frac{\theta}{2}\int_{f^{-1}(B_\rho(x_3))}
                               \vn{\nabla_{(2)}A}^2d\mu
  + c_\rho^2\frac{2}{\theta}\int_{f^{-1}(B_\rho(x_3))}
   \vn{A}^2d\mu
   + c\int_M \vn{A}^{2k-2}d\mu
\\
&\le \frac{\theta}{2}\int_{f^{-1}(B_\rho(x_3))}
                               \vn{\nabla_{(2)}A}^2d\mu
   + c(\theta,\rho,n,C_{A\!B})\Big[\big(\delta_0^{(m)}\big)^{2n+2+\frac{2}{m}}
   + \big(\delta_0^{(m)}\big)^{\frac{2k-2}{m}}\Big].
\end{align*}
Note that in any case, the exponent of $\delta_0^{(m)}$ is greater than 1 due to the conditions on $m$.
This gives the first part of the lemma.

If $j=k=l=0$ then obviously 
\[
h \le 
  c(\rho,n,C_{A\!B})\big(\delta_0^{(m)}\big)^{n+1}.
\]
This finishes the proof.
\end{proof}

\begin{rmk}
In the special case where $h = \hH$ and $n=2$, the estimate reads
\[
\hH
\le \theta\int_{f^{-1}(B_\rho(x_1))}\vn{\nabla_{(2)}A}^2d\mu
 + \frac{c_{B\!Z}}{4\theta\sqrt{\text{Vol }
M_0}}\bigg(\frac{16\sqrt{3}}{\rho\pi}\bigg)^24|M_0|^{\frac{37}{2}}\big(\delta_0^{(4)}\big)^\frac{13}{2},
\]
where $c_{B\!Z}$ is the constant from the inequality \eqref{inBZ}.
\end{rmk}

We are now ready to prove Theorem \ref{thmapriorih} as essentially a corollary to Lemma
\ref{lemalltheworkapriorih} above.

\begin{proof}[Proof of Theorem \ref{thmapriorih}]
First note that
\begin{align*}
\int_{f^{-1}(B_{2\rho}(x))}\big(\vn{A}^4+\vn{A}^2\big)d\mu
 &\le \sup_{x^*\in B_{2\rho}(x)} 4^{n+1} \int_{f^{-1}(B_\rho(x^*))}\big(\vn{A}^4+\vn{A}^2\big)d\mu
\\*
 &\le 4^{n+1}C_{A\!B}^{1-\frac{4}{m}}\big(\delta_0^{(m)}\big)^{\frac{4}{m}}.
\end{align*}
By Lemma \ref{lemalltheworkapriorih} we are now finished, choosing
\[
\theta = \frac{\theta^*}{4^{n+1}C^{1-\frac{4}{m}}_{A\!B}\big(\delta_0^{(m)}\big)^\frac{4}{m}}.
\]
\end{proof}

\begin{rmk}
In each of the previous inequalities we have been primarily concerned with integrals localised
to a ball $f^{-1}(B_\rho(x))$.  In the following sections where we derive the basic integral
estimates, the domain of integration will instead be the set $[\gamma>0]$,
for $\gamma$ as in \eqref{e:gamma}.  This is necessary to not only obtain the local
integral estimates, but also to allow us enough freedom to choose various appropriate $\gamma$
functions, depending upon the situation.  To bridge the gap between the two domains of integration
we may choose $\gamma = \tilde{\gamma} \circ f$ to be such that
\[
\chi_{B_\rho(x)} \le \tilde{\gamma} \le \chi_{B_{2\rho}(x)}
\]
and $\gamma \in C^2(M)$.  Then for a non-negative integrand we crudely estimate
\[
\int_{f^{-1}(B_\rho(x))} [\cdots] d\mu \le \int_{[\gamma>0]} [\cdots] d\mu
 \le \int_{f^{-1}(B_{2\rho}(x))} [\cdots] d\mu.
\]
This is why in Theorem \ref{thmapriorih} we see integrals with balls of radii $2\rho$ on the left.
\end{rmk}

Theorem 4 gives us the opportunity to obtain the derivative of curvature estimates in the ball
$B_\rho(x_1)$, but nowhere else.  This is not enough to prove the lifespan theorem.  However, we may
still proceed by using the estimates in the ball $B_\rho(x_1)$ to bound the constraint function over
all of $M_t$, and then once this is accomplished we can go back and prove the required derivative of
curvature estimates everywhere else on $M_t$.

\begin{cor}[The curvature estimates on a special ball]
\label{corestimatesinaspecialball}
Suppose $n\in\{2,3\}$ and let $f:M^n\times[0,T^*]\rightarrow\R^{n+1}$ be a \eqref{CSD} flow with $h$
satisfying the assumptions of Theorem \ref{thmapriorih}.  Then there is a $\delta_0^{(m)} =
\delta_0^{(m)}(n,M_0)$ such that if
\begin{equation*}
\sup_{t\in[0,T^*],x\in\R^{n+1}}\int_{f^{-1}(B_\rho(x))}\vn{A}^md\mu\le\delta_0^{(m)},
\end{equation*}
there is an $x_1\in\R^{n+1}$ such that
\begin{equation*}
\vn{\nabla_{(2)}A}^2_{\infty,f^{-1}(B_\rho(x_1))}
 \le c\big(\delta_0^{(m)},T^*,C_{A\!B},\rho,j,k,l,m,\alpha_0(2)\big),
\end{equation*}
where $\displaystyle \alpha_0(p) = \sum_{j=0}^p \sup_{x\in\R^{n+1}}
\vn{\nabla_{(j)}A}_{2,f^{-1}(B_\rho(x))}\bigg|_{t=0}$.
\label{curvestspecialball}
\end{cor}
\begin{proof}
Observe that the smallness assumption and \eqref{AB} implies that
\begin{align*}
\int_{f^{-1}(B_\rho(x))} \vn{A}^n d\mu
 &\le C_{A\!B\!}^\frac{m-n}{m}\bigg(\int_{f^{-1}(B_\rho(x)} \vn{A}^m d\mu\bigg)^\frac{n}{m}
 &\le  C_{A\!B\!}^\frac{m-n}{m}\big(\delta_0^{(m)}\big)^\frac{n}{m}
 &< \epsilon_0,
\end{align*}
for 
\[
\delta_0^{(m)} < (\epsilon_0)^\frac{m}{n}C_{A\!B}^\frac{n-m}{m}.
\]
Let $\gamma$ be a cutoff function on $M$ between a ball of radius $\rho$ and a ball of radius
$2\rho$, as in the remark above.  Then the smallness assumption \eqref{eq9} of Proposition
\ref{p:prop44} is satisfied for $\delta_0^{(m)}$ as above, that is
\[
\sup_{[0,T^*]}\int_{f^{-1}(B_\rho(x))}\vn{A}^nd\mu\le\epsilon_0.
\]
Proposition \ref{EQpropBEwithh} with $k=0$ and our choice of $\gamma$ gives:
\begin{align*}
&\rD{}{t}\int_{f^{-1}(B_\rho(x))} \vn{A}^2d\mu
 + (2-\theta)\int_{f^{-1}(B_\rho(x))}\vn{\nabla_{(2)}A}^2d\mu\\*
&\qquad \le 
             ch\int_{f^{-1}(B_{2\rho}(x))}\left([A*A]*A\right)
                     d\mu
            + ch\int_{f^{-1}(B_{2\rho}(x))}\vn{A}^2
                     d\mu
\notag\\*&\qquad\qquad\qquad
 + c\int_{f^{-1}(B_{2\rho}(x))} \vn{A}^2d\mu
 + c\int_{f^{-1}(B_{2\rho}(x))} \left([P_3^{2}(A)+P_5^0(A)]*A\right) d\mu.
\end{align*}
Using Theorem \ref{thmapriorih} we obtain
\begin{align*}
&\rD{}{t}\int_{f^{-1}(B_\rho(x_1))} \vn{A}^2d\mu
 + (2-\theta)\int_{f^{-1}(B_\rho(x_1))}\vn{\nabla_{(2)}A}^2d\mu\\*
&\qquad \le 
  c\delta_0^{(m)}
 + c\int_{f^{-1}(B_{2\rho}(x_1))} \left([P_3^{2}(A)+P_5^0(A)]*A\right) d\mu.
\end{align*}
Proceeding now exactly as in Proposition \ref{p:prop44}, we recover \eqref{eq10} for balls centred
at the point $x_1$.  Note that no constant depends on $\vn{h}_\infty$.  Moving on, we use
the equation above to conclude \eqref{e:prop45} in the case where there are no derivatives of
curvature, with no additional factors of the constraint function on the right hand side.  That is,
\begin{equation*}
  \begin{split}
&\rD{}{t}\int_{f^{-1}(B_\rho(x_1))} \vn{A}^2 d\mu
 + \frac{1}{2}\int_{f^{-1}(B_\rho(x_1))} \vn{\nabla_{(2)}A}^2 d\mu \\
&\qquad\qquad\qquad 
  \le 
       c\vn{A}^2_{2,f^{-1}(B_{2\rho}(x_1))}(1+\vn{A}^4_{\infty,f^{-1}(B_{2\rho}(x_1))}).
  \end{split}
\end{equation*}
Using this in the proof of Proposition \ref{p:prop46} in place of Proposition \ref{p:prop45} gives
the required derivative of curvature bounds.
\end{proof}

\begin{rmk}Allowable choices of $x_1$ depend upon the splitting of integrals in Lemma
\ref{lemalltheworkapriorih}, and this depends upon $j,k$ and $l$.  The proof of the next result
will depend upon which class of allowable points is associated with the given constraint function.
\end{rmk}

We note that the assumption required is global, disguised as a local assumption.  This is different
to the case where we have no constraint function (such as for the surface diffusion or Willmore
flows).  However, even there, in the final argument used to prove the lifespan theorem one still
requires this `global disguised as local' assumption.  We are merely introducing this concept
earlier in the analysis.

\begin{cor}[The uniform bound for $h$]
Suppose $n\in\{2,3\}$ and let $f:M^n\times[0,T^*]\rightarrow\R^{n+1}$ be a \eqref{CSD} flow with $h$
satisfying the assumptions of Theorem \ref{thmapriorih}.  Then there is a $\delta_0^{(m)} =
\delta_0^{(m)}(n,M_0)$ such that if 
\begin{equation}
\sup_{[0,T^*],x\in\R^{n+1}}\int_{f^{-1}(B_\rho(x))}\vn{A}^md\mu\le\delta_0^{(m)},
\label{ERRATAeqnsmallness}
\end{equation}
the constraint function satisfies the estimate
\[
\vn{h}_{[0,T^*],\infty} \le c_h < \infty,
\]
where $c_h = c_h(\delta^{(m)}_0,C_{A\!B},\rho,j,k,l,n)$.
\end{cor}
\begin{proof}
Using Corollary \ref{curvestspecialball} above, we can directly estimate $h$ by localising as in the
proof of Lemma \ref{lemalltheworkapriorih}.  This is however contingent upon us retrieving integrals
around an allowable point $x_1\in\R^{n+1}$ from the conclusion of Corollary
\ref{curvestspecialball}.  So we must be somewhat careful with our estimates below.

Firstly, for the case where $j\ge\max\{2,2k+1\}$,
\begin{align*}
h &\le \int_M P_j^2(A) + P_k^1(A) + P_l^0(A) d\mu
\\
&\le c_\rho c\int_{f^{-1}(B_\rho(x_1))}
                               \vn{\nabla_{(2)}A}\cdot\vn{A}^{j-1}d\mu
 + \int_M\vn{\nabla A}^2\vn{A}^{j-2} d\mu
\\*
&\qquad + c\int_M \vn{\nabla A}\cdot\vn{A}^{k-1} d\mu
 + c\int_M \vn{A}^l d\mu
\\
&\le \frac{1}{2}\int_{f^{-1}(B_\rho(x_1))} \vn{\nabla_{(2)}A}^2d\mu
 + c\Big(1+(j-2)\frac{4-j}{j-3}\Big)\int_M \IP{A}{\Delta A}\vn{A}^{j-2} d\mu
\\
&\qquad
 + c_\rho^2\frac{c}{2}\int_{f^{-1}(B_\rho(x_1))} \vn{A}^{2j-2}d\mu
 + c(j,k,l)\int_M \vn{A}^{2k-j} + \vn{A}^l d\mu
\\
&\le \frac{1}{2}\int_{f^{-1}(B_\rho(x_1))} \vn{\nabla_{(2)}A}^2d\mu
 + c(j)\int_M \vn{\nabla_{(2)} A}\cdot\vn{A}^{j-1} d\mu
\\*
&\qquad
 + c_\rho^2\frac{c}{2}\int_{f^{-1}(B_\rho(x_1))} \vn{A}^{2j-2}d\mu
 + c(j,k,l)\int_M \vn{A}^{2k-j} + \vn{A}^l d\mu
\\
&\le \int_{f^{-1}(B_\rho(x_1))} \vn{\nabla_{(2)}A}^2d\mu
 + c_\rho^2c(j)\int_{f^{-1}(B_\rho(x_1))} \vn{A}^{2j-2}d\mu
\\*
&\qquad
 + c(j,k,l)\int_M \vn{A}^{2k-j} + \vn{A}^l d\mu
\\
&\le \int_{f^{-1}(B_\rho(x_1))} \vn{\nabla_{(2)}A}^2d\mu
 + c_\rho^2c(j,C_{A\!B})\Big(\int_{f^{-1}(B_\rho(x_1))} \vn{A}^md\mu\Big)^\frac{2j-2}{m}
\\*
&\ 
 + c(j,k,l,C_{A\!B})\sup_{x\in\R^{n+1}}c_\rho
                     \bigg[\Big(\int_{f^{-1}(B_\rho(x))} \vn{A}^{m}d\mu\Big)^\frac{2k-j}{m}
                   + \Big(\int_{f^{-1}(B_\rho(x))} \vn{A}^m d\mu\Big)^\frac{l}{m}\bigg]
\\
&\le c_h(\delta^{(m)}_0,C_{A\!B},\rho,j,k,l,n) < \infty.
\end{align*}
The other cases are simpler, and estimated as in Lemma \ref{lemalltheworkapriorih}, finished off
using Corollary \ref{curvestspecialball} as above.
\end{proof}

This shows that for the class of constraint functions satisfying the conditions of Theorem
\ref{thmapriorih} and a small curvature condition \eqref{ERRATAeqnsmallness}, the a priori bound
\eqref{A1} holds.  Since we only require \eqref{A1} while \eqref{ERRATAeqnsmallness} is true,
this is enough to include constraint functions satisfying 
the growth condition \eqref{ERRATAgrowthoncosntraintcondition} and area bound \eqref{AB}
in our main theorem.


\begin{rmk}  There is an alternative approach, based also on Theorem \ref{topping}, which works
without the assumption \eqref{AB}.  However this requires monotonicity of $\int |H|$ on a ball
around $x_1$, and does not give higher dimensional results.
It is relevant to $h_K$ flow, where we have monotonicity of $\int H$ on the entire manifold,
for all time.  However the essential problem is that there is no known condition which rules out the
case where mean curvature is becoming more negative in one part of the manifold and more positive
in another part, such that the integral over the entire manifold is non-increasing, but for any
small ball the integral $\int |H|$ is increasing.
Also, even if such a case is ruled out, we have no way of ensuring that the special points
$x_1$ are in the regions of  $M$ where $\int |H|$ is monotone.  What we really lack is a non-trivial 
condition we can impose on $M_0$ such that monotonicity of $\int H$ implies monotonicity of $\int
|H|$, however without the maximum principle we have not been able to achieve this.  Thus $h_K$ still
presents difficulty.
\end{rmk}

\section{Evolution equations for integrals of curvature.}

To begin, we state the following elementary evolution equations, whose proof is standard.
\begin{lem}\label{LemEV1}
  For $f:M^n\times[0,T)\rightarrow\R^{n+1}$ evolving by \eqref{CSD} the following equations
hold:
\label{l:curvevo}
\begin{align*}
  \pD{}{t}g &= 2(\Delta H)A + 2hA,\\
  \pD{}{t}d\mu &= (Hh+H\Delta H)d\mu,\\
  \pD{}{t}\nu &= -\nabla \Delta H \text{,  and}\\
  \pD{}{t}A &= -\Delta^2A + P_3^2(A) + hA*A.\\
\end{align*}
\end{lem}
\begin{lem}Let $f:M^n\times[0,T)\rightarrow\R^{n+1}$ be a \eqref{CSD} flow.  Then the following
equation holds:
\[
  \pD{}{t}\nabla_{(k)}A = -\Delta^2\nabla_{(k)}A + hP_2^k(A) + P_3^{k+2}(A).
\]
\end{lem}
The following is an easy consequence of the above lemma.
\begin{cor}
Let $f:M^n\times[0,T)\rightarrow\R^{n+1}$ be a \eqref{CSD} flow.  Then the following equation
holds:
\[
  \pD{}{t}\vn{\nabla_{(k)}A}^2 = - 2\IP{\nabla_{(k)}A}{\nabla^p\Delta\nabla_p\nabla_{(k)}A}
                                 + [hP_2^k(A) + P_3^{k+2}(A)]*\nabla_{(k)}A.
\]
\label{S4cor2}
\end{cor}
%
Using Corollary \ref{S4cor2}, we derive the following integral identity.
\begin{cor}
\label{CorRE2}
Let $f:M^n\times[0,T)\rightarrow\R^{n+1}$ be a \eqref{CSD} flow, and $\gamma$ as in \eqref{e:gamma}.
Then for any $s\ge0$,
\begin{align*}
\rD{}{t}\int_M \vn{\nabla_{(k)}A}^2\gamma^sd\mu + 2\int_M\vn{\nabla_{(k+2)}A}^2\gamma^sd\mu
   &=  \int_M \vn{\nabla_{(k)}A}^2(\partial_t\gamma^s)d\mu
\\ &\hskip-6.5cm
    + 2\int_M\IP{(\nabla\gamma^s)(\nabla_{(k)}A)}{\Delta\nabla_{(k+1)}A} d\mu
     - 2\int_M\IP{(\nabla\gamma^s)(\nabla_{(k+1)}A)}{\nabla_{(k+2)}A}d\mu 
\\ &\hskip-6.5cm
    + \int_M\gamma^s[(P_3^{k+2}(A)+hP_2^k(A))*\nabla_{(k)}A]d\mu.
\end{align*}
\end{cor}
We now wish to use interpolation to estimate the extraneous terms from integration by parts.  For
$k=1$, the required inequality follows easily (for $\theta,\beta>0$):
\begin{equation}
(1-\beta)\int_M \vn{\nabla A}^2\gamma^{s-2}d\mu
\le \theta\int_M\vn{\nabla_{(2)}A}^2\gamma^sd\mu
   + \frac{\beta+\theta[(s-2)c_{\gamma1}]^2}{4\beta\theta}
     \int_M\vn{A}^2\gamma^{s-4}d\mu.
\label{babyint}
\end{equation}

For $k>1$ however we need a more powerful version of the above.  Let $2\le p < \infty$, $k\in\N$,
$s\ge kp$, and $\theta > 0$.  Then we have
\begin{equation}
\left( \int_M \vn{\nabla_{(k)}A}^p\gamma^sd\mu \right)^\frac{1}{p}
  \le \theta \left( \int_M \vn{\nabla_{(k+1)}A}^p\gamma^{s+p}d\mu \right)^\frac{1}{p}
      + c \left( \int_{[\gamma>0]}\vn{A}^p\gamma^{s-kp}d\mu \right)^\frac{1}{p},
\label{TMint1}
\end{equation}
where $c = c(\theta, c_{\gamma1}, s, p)$.
This is proved, essentially, by induction on the inequality \eqref{babyint}.  Details can be found in
\cite{kuwert2002gfw}. 
  We now estimate the equality in Corollary \ref{CorRE2}.



\begin{prop}
\label{p:prop3}
Let $f:M^n\times[0,T)\rightarrow\R^{n+1}$ be a \eqref{CSD} flow with $h$ satisfying \eqref{A1}
and $\gamma$ a cutoff function as in \eqref{e:gamma}.  Then for a fixed $\theta > 0$ and
$s\ge2k+4$,
\begin{align}
&\rD{}{t}\int_M \vn{\nabla_{(k)}A}^2\gamma^sd\mu
 + (2-\theta)\int_M \vn{\nabla_{(k+2)}A}^2\gamma^sd\mu\notag\\*
&\qquad \le (c+ch)\int_M \vn{A}^2\gamma^{s-4-2k}d\mu
            + ch\int_M\left(\nabla_{(k)}[A*A]*\nabla_{(k)}A\right)
                     \gamma^{s} d\mu
\notag\\*&\qquad\qquad\qquad
            + c\int_M \left([P_3^{k+2}(A)+P_5^k(A)]*\nabla_{(k)}A\right)\gamma^{s} d\mu,
\notag
\end{align}
where $c=c(c_{\gamma1},c_{\gamma2},s,k,\vn{h}_{\infty,[0,T)},\theta)$.
\end{prop}
The proof is standard, and follows by using Corollary \ref{CorRE2} and inequality
\eqref{TMint1} to deal with the derivatives of $\gamma$, estimating the result, and absorbing.
To prove Corollary \ref{curvestspecialball} we also need a version of the above estimate where we do
not assume \eqref{A1}.  For this purpose, we state the following proposition.

\begin{prop}
\label{EQpropBEwithh}
Let $f:M\times[0,T)\rightarrow\R^3$ be a \eqref{CSD} flow 
and $\gamma$ a cutoff function as in \eqref{e:gamma}.  Then for a fixed $\theta > 0$ and
$s\ge2k+4$,
\begin{align*}
&\rD{}{t}\int_M \vn{\nabla_{(k)}A}^2\gamma^sd\mu
 + (2-\theta)\int_M \vn{\nabla_{(k+2)}A}^2\gamma^sd\mu\\*
&\qquad \le 
             ch\int_M\left(\nabla_{(k)}[A*A]*\nabla_{(k)}A\right)
                     \gamma^{s} d\mu
            + ch\int_M\vn{\nabla_{(k)}A}^2
                     \gamma^{s-1} d\mu
\\*&\qquad\qquad\qquad
 + c\int_M \vn{A}^2\gamma^{s-4-2k}d\mu
 + c\int_M \left([P_3^{k+2}(A)+P_5^k(A)]*\nabla_{(k)}A\right)\gamma^{s} d\mu,
\end{align*}
where $c=c(c_{\gamma1},c_{\gamma2},s,k)$.
\end{prop}

\section{Integral estimates with small concentration of curvature.}

We will first need a few Sobolev and interpolation inequalities, importantly the Michael-Simon
Sobolev inequality, \cite{michael1973sam}.

\begin{thm}[Michael-Simon Sobolev inequality] Let $f:M^n\rightarrow\R^{n+1}$ be a smooth immersion.  Then
for any $u\in C_c^1(M)$ we have
\[  \left(\int_M |u|^{n/(n-1)}d\mu\right)^{(n-1)/n}
    \le \frac{4^{n+1}}{\omega_n^{1/n}} \int_M \vn{\nabla u} + |u|\ |H| d\mu,
\]
where $\omega_n$ is the volume of the unit ball in $\R^n$.
\end{thm}

The eventual goal for this
section is to prove local $L^\infty$ estimates for all derivatives of curvature.  Our main tool to
convert $L^p$ bounds to $L^\infty$ bounds is the following theorem, which is an $n$-dimensional
analogue of Theorem 5.6 from \cite{kuwert2002gfw}.  Its proof may be found in \cite{mythesis}.

\begin{thm}
\label{myLZthm}
Let $f:M^n\rightarrow\R^{n+1}$ be a smooth immersed hypersurface. For $u\in C_c^1(M)$,
$n<p\le\infty$, $0\le \beta\le \infty$ and $0<\alpha\le 1$ where $\frac{1}{\alpha} =
\big(\frac{1}{n}-\frac{1}{p}\big)\beta + 1$ we have
\begin{equation}
  \vn{u}_\infty \le c\vn{u}_\beta^{1-\alpha}(\vn{\nabla u}_p + \vn{Hu}_p)^\alpha,
\label{myLZthmeqn}
\end{equation}
where $c = c(n,p,\beta)$.
\end{thm}

The proof follows ideas from \cite{ladyzhenskaya1968laq} and \cite{kuwert2002gfw}; see also Section
6 of \cite{mantegazza2002sge}.  Due to the
exponent in the Michael-Simon Sobolev inequality, it is not possible to decrease the lower bound on
$p$, even at the expense of other parameters in the inequality.
%
%
This introduces a restriction on the dimension of our immersion, and 
is highlighted in the following local refinement to Theorem
\ref{myLZthm}.

\begin{prop}
\label{MS2prop}
 Let $n\in\{2,3\}$.  Then for any tensor $T$ on $f:M^n\rightarrow\R^{n+1}$ and $\gamma$
as in \eqref{e:gamma},
\begin{equation}
  \vn{T}^4_{\infty,[\gamma=1]}
    \le c\vn{T}^{4-n}_{2,[\gamma>0]}\big( \vn{\nabla_{(2)}T}^n_{2,[\gamma>0]}
                                   + \vn{TA^2}^n_{2,[\gamma>0]}
                                   + \vn{T}^n_{2,[\gamma>0]}\big),
\label{MS2}
\end{equation}
where $c=c(c_{\gamma1},n)$.
Assume $T=A$, and if $n=3$ also assume \eqref{AB}.  
Then there exists an $\epsilon_0 = \epsilon_0(c_{\gamma1},c_{\gamma2},n)$ such that if 
\[
  \vn{A}^n_{n,[\gamma>0]} \le \epsilon_0
\]
we have
\begin{equation}
  \vn{A}^{8n-12}_{\infty,[\gamma=1]}
    \le c\epsilon_0
         \big(\vn{\nabla_{(2)}A}^{2n^2-3n}_{2,[\gamma>0]}
           + \epsilon_0\big),
\label{MS2secondstatement}
\end{equation}
with $c=c(c_{\gamma1},c_{\gamma2},n,\epsilon_0)$ for $n=2$ and
$c=c(c_{\gamma1},c_{\gamma2},n,\epsilon_0,C_{A\!B})$ for $n=3$.
\end{prop}

The proof is similar to that of Lemma 4.3 in \cite{kuwert2002gfw}, except for the $n=3$ case.  While
the first statement follows a similar proof with minor alterations, for the second statement 
one needs to use the $n=3$ version of the below multiplicative Sobolev inequality, and the area
bound \eqref{AB} with H\"older's inequality.




\begin{lem}
\label{MS1lem}
Let $\gamma$ be as in \eqref{e:gamma}.  Then for an immersed surface $f:M^2\rightarrow\R^{3}$ we
have
  \begin{align*}
  \int_M\vn{A}^6\gamma^sd\mu + \int_M\vn{A}^2\vn{\nabla A}^2\gamma^sd\mu
   &\le  c\int_{[\gamma>0]}\vn{A}^2d\mu\int_M(\vn{\nabla_{(2)}A}^2 + \vn{A}^6)\gamma^sd\mu
   \\
   &\hskip+2cm + c(c_{\gamma1})^4\Big( \int_{[\gamma>0]}\vn{A}^2d\mu \Big)^2,
  \end{align*}
and for an immersion $f:M^3\rightarrow\R^4$,
  \begin{align*}
\int_M\vn{A}^6\gamma^sd\mu + \int_M\vn{A}^2\vn{\nabla A}^2\gamma^sd\mu
 &\le \theta\int_M \vn{\nabla_{(2)}A}^2\gamma^sd\mu
\\
&\hskip-4.5cm
    +  c\vn{A}_{3,[\gamma>0]}^\frac{3}{2}
     \int_M \big(\vn{\nabla_{(2)} A}^2 + \vn{A}^6\big) \gamma^s d\mu
    + c(c_{\gamma1})^3\big(\vn{A}^3_{3,[\gamma>0]} + \vn{A}^\frac{9}{2}_{3,[\gamma>0]}\big),
  \end{align*}
where $\theta\in(0,\infty)$ and $c = c(s,\theta)$ is an absolute constant.
\end{lem}
\begin{proof}
The first statement is Lemma 4.2 in \cite{kuwert2002gfw}.
For the second, first observe that
\begin{align*}
\int \vn{\nabla A}^3\gamma^sd\mu 
&\le 
 \int_M\big(\IP{A}{\Delta A}*{\nabla A} + A*\nabla A*\nabla\vn{\nabla A}\big) \gamma^sd\mu
\\
&\qquad
    + s\int_M \big(A*\nabla A*\nabla A*\nabla\gamma\big)\gamma^{s-1}d\mu
\\
 &\le \frac{1}{4\theta}\int_M\vn{\nabla_{(2)}A}^2\gamma^sd\mu
    + \theta\int_M\vn{A}^2\vn{\nabla A}^2\gamma^sd\mu
\\
&\qquad
    + \frac{(sc_{\gamma1})^34^2}{3}\int_M \vn{A}^3\gamma^{2s-3}d\mu
    + \frac{1}{6}\int_M \vn{\nabla A}^3\gamma^{s}d\mu
\\
&\le \frac{1}{4\theta}\int_M \vn{\nabla_{(2)}A}^2\gamma^sd\mu 
    + \frac{\theta^3}{3}\int_M \vn{A}^6\gamma^sd\mu
    + \frac{(sc_{\gamma1})^34^2}{3}\int_M \vn{A}^3\gamma^{2s-3}d\mu
\\
&\qquad
    + \frac{5}{6}\int_M \vn{\nabla A}^3 \gamma^sd\mu,
\intertext{so}
\int \vn{\nabla A}^3d\mu 
&\le \frac{3}{2\theta}\int_M \vn{\nabla_{(2)}A}^2\gamma^sd\mu + 2\theta^3\int_M \vn{A}^6\gamma^sd\mu
    + 2(sc_{\gamma1})^34^2\int_{[\gamma>0]} \vn{A}^3d\mu,
\end{align*}
for any $\theta \in (0,\infty)$.

Now we use the Michael-Simon Sobolev inequality with $u=\vn{A}^4\gamma^{2s/3}$ to estimate
\begin{align*}
\Big(\int_M \vn{A}^6 \gamma^sd\mu\Big)^\frac{2}{3}
 &\le c\int_M \vn{A}^3\vn{\nabla A}\gamma^\frac{2s}{3}d\mu
    + c\int_M \vn{A}^4\vn{\nabla\gamma}\gamma^{\frac{2s-3}{3}}d\mu
    + c\int_M \vn{A}^5\gamma^\frac{2s}{3} d\mu
\\
 &\le c\int_M \vn{\nabla A}^2\vn{A}\gamma^sd\mu + c\int_M \vn{A}^5\gamma^s d\mu
    + c(c_{\gamma1})^2\vn{A}^3_{3,[\gamma>0]}
\\
 &\le c\int_M \vn{\nabla A}^2\vn{A}\gamma^sd\mu + 
                \Big(\int_M\vn{A}^6\gamma^sd\mu\Big)^\frac{2}{3}
                \Big(\int_{[\gamma>0]}\vn{A}^{3}d\mu\Big)^\frac{1}{3}
\\*
&\qquad
    + c(c_{\gamma1})^2\vn{A}^3_{3,[\gamma>0]},
\intertext{so}
\int_M \vn{A}^6 \gamma^s d\mu
 &\le c\Big(\int_M \vn{\nabla A}^2\vn{A}\gamma^sd\mu\Big)^\frac{3}{2}
    + c\vn{A}_{3,[\gamma>0]}^\frac{3}{2}\int_M \vn{A}^6 \gamma^s d\mu
\\*
&\qquad
    + c(c_{\gamma1})^3\vn{A}^\frac{9}{2}_{3,[\gamma>0]}
\\
 &\le c\vn{A}_{3,[\gamma>0]}^\frac{3}{2}
     \int_M \big(\vn{\nabla_{(2)} A}^2 + \vn{A}^6\big) \gamma^s d\mu
    + c(c_{\gamma1})^3\vn{A}^\frac{9}{2}_{3,[\gamma>0]}.
\end{align*}
This estimates the first term.  For the second, we can employ a more direct technique using our
estimates above,
\begin{align*}
\int_M\vn{A}^2\vn{\nabla A}^2\gamma^sd\mu
 &\le c\int_M \vn{A}^6 \gamma^s d\mu + c\int_M \vn{\nabla A}^3\gamma^s d\mu
\\
 &\le \theta\int_M \vn{\nabla_{(2)}A}^2\gamma^sd\mu
    +  c_\theta\vn{A}_{3,[\gamma>0]}^\frac{3}{2}
     \int_M \big(\vn{\nabla_{(2)} A}^2 + \vn{A}^6\big) \gamma^s d\mu
\\
&\qquad
    + c_\theta(c_{\gamma1})^3\big(\vn{A}^3_{3,[\gamma>0]} + \vn{A}^\frac{9}{2}_{3,[\gamma>0]}\big).
\end{align*}
This estimates the second term, and combining the two estimates above finishes the proof.
\end{proof}

The proposition used for the constructive part of the final argument can now be proved.

\begin{prop} 
\label{p:prop44}
Let $n\in\{2,3\}$.  Suppose $f:M^n\times[0,T^*]\rightarrow\R^{n+1}$ is a \eqref{CSD} flow with $h$
satisfying \eqref{A1} and $\gamma$ a cutoff function as in \eqref{e:gamma}.
  Additionally, if $n=3$ assume \eqref{AB}.
Then there is an $\epsilon_0 =
\epsilon_0\big(c_{\gamma1},c_{\gamma2},\vn{h}_{\infty,[0,T^*]}\big)$ such that if
\begin{equation}
\label{eq9}
\epsilon = \sup_{[0,T^*]}\int_{[\gamma>0]}\vn{A}^nd\mu\le\epsilon_0
\end{equation}
then for any $t\in[0,T^*]$ we have
\begin{equation}
\label{eq10}
  \begin{split}
&\int_{[\gamma=1]} \vn{A}^2 d\mu + \int_0^t\int_{[\gamma=1]} (\vn{\nabla_{(2)}A}^2 +
\vn{A}^2\vn{\nabla A}^2 + \vn{A}^6) d\mu d\tau \\
&\qquad\qquad\qquad \le \int_{[\gamma > 0]} \vn{A}^2 d\mu\Big|_{t=0}
 + c\epsilon^\frac{2}{n} t,
  \end{split}
\end{equation}
where $c=c\big(c_{\gamma1},c_{\gamma2},\vn{h}_{\infty,[0,T^*]},C_{A\!B}\big)$.
\end{prop}
\begin{proof}
For $n=2$, similar to \cite{kuwert2002gfw}, except for the extra integrals arising from
the constraint function.  These are dealt with using \eqref{A1} and absorbing.  The details are
similar to the $n=3$ case, which we will describe below.
Setting $k=0$ and $s=4$ in Proposition \ref{p:prop3} we have
\begin{align}
&\rD{}{t}\int_M \vn{A}^2\gamma^4d\mu
 + (2-\theta)\int_M \vn{\nabla_{(2)}A}^2\gamma^4d\mu
 \le (c+ch)\int_{[\gamma>0]} \vn{A}^2d\mu
\notag\\
&\qquad
            + ch\int_M\left([A*A]*A\right)
                     \gamma^{4} d\mu
\label{prop44neweq4}
           + c\int_M \left([P_3^{2}(A)+P_5^0(A)]*A\right)\gamma^{4} d\mu.
\end{align}
First we estimate the $P$-style terms:
\begin{align}
\int_M&\left([P_3^{2}(A)+P_5^0(A)]*A\right)\gamma^{4} d\mu
\notag\\&\le
     c\int_M\big[\vn{A}^3\cdot\vn{\nabla_{(2)}A}
                + \vn{\nabla A}^2\cdot\vn{A}^2+\vn{A}^6\big]\gamma^{4} d\mu
\notag\\&\le
     \theta\int_M \vn{\nabla_{(2)}A}^2\gamma^4d\mu
   +  c\int_M (\vn{A}^6 + \vn{\nabla A}^2\vn{A}^2)\gamma^{4} d\mu.
\notag
\intertext{
We use Lemma \ref{MS1lem} to estimate the second integral and obtain (recall $n=3$)} 
\int_M&\left([P_3^{2}(A)+P_5^0(A)]*A\right)\gamma^{4} d\mu
\notag\\*
&\le
     \theta\int_M \vn{\nabla_{(2)}A}^2\gamma^4d\mu
     + c\vn{A}_{3,[\gamma>0]}^\frac{3}{2}\int_M(\vn{\nabla_{(2)}A}^2 + \vn{A}^6)\gamma^4d\mu
\notag\\*&\qquad
     + c\big(c_{\gamma1}\big)^3\big(\vn{A}^3_{3,[\gamma>0]} + \vn{A}^\frac{9}{2}_{3,[\gamma>0]}\big).
\label{prop44neweq1}
\end{align}
We add the integrals $\int_M \vn{A}^6 \gamma^4d\mu$ and $\int_M \vn{\nabla A}^2\vn{A}^2\gamma^4d\mu$
to the estimate
\eqref{prop44neweq4} and obtain
\begin{align*}
&\rD{}{t}\int_M \vn{A}^2\gamma^4d\mu
 + (2-\theta)\int_M \big(\vn{\nabla_{(2)}A}^2+\vn{A}^2\vn{\nabla A}^2+\vn{A}^6\big)\gamma^4 d\mu
\\
&\qquad \le (c+ch)\int_{[\gamma>0]} \vn{A}^2d\mu
            + ch\int_M\left([A*A]*A\right)
                     \gamma^{4} d\mu
\\
&\qquad\qquad
     + c\int_M \big(\vn{A}^2\vn{\nabla A}^2+\vn{A}^6\big)\gamma^4 d\mu
     + c\int_M \left([P_3^{2}(A)+P_5^0(A)]*A\right)\gamma^{4} d\mu
\\
&\qquad \le c(1+h^2)\int_{[\gamma>0]} \vn{A}^2d\mu
     + c\int_M 
        \big(\vn{A}^3\vn{\nabla_{(2)}A}+\vn{A}^2\vn{\nabla A}^2+\vn{A}^6\big)\gamma^4 d\mu.
\intertext{We now use \eqref{prop44neweq1} to obtain}
&\rD{}{t}\int_M \vn{A}^2\gamma^4d\mu
 + (2-\theta)\int_M \big(\vn{\nabla_{(2)}A}^2+\vn{A}^2\vn{\nabla A}^2+\vn{A}^6\big)\gamma^4 d\mu
\\*
&\qquad \le c(1+h^2)\int_{[\gamma>0]} \vn{A}^2d\mu
     + \theta\int_M\vn{\nabla_{(2)}A}^2\gamma^4d\mu
\\*
&\qquad\qquad
     + c\vn{A}_{3,[\gamma>0]}^\frac{3}{2}\int_M(\vn{\nabla_{(2)}A}^2 + \vn{A}^6)\gamma^4d\mu
\\*&\qquad\qquad
     + c\big(c_{\gamma1}\big)^3\big(\vn{A}^3_{3,[\gamma>0]} + \vn{A}^\frac{9}{2}_{3,[\gamma>0]}\big).
\\*
&\qquad \le c(1+h^2)C_{A\!B}^\frac{1}{3}\vn{A}_{3,[\gamma>0]}^2
     + \theta\int_M\vn{\nabla_{(2)}A}^2\gamma^4d\mu
\\*
&\qquad\qquad
     + c\vn{A}_{3,[\gamma>0]}^\frac{3}{2}\int_M(\vn{\nabla_{(2)}A}^2 + \vn{A}^6)\gamma^4d\mu
\\*&\qquad\qquad
     + c\big(c_{\gamma1}\big)^3\big(\vn{A}^3_{3,[\gamma>0]} + \vn{A}^\frac{9}{2}_{3,[\gamma>0]}\big).
\intertext{Absorbing,}
&\rD{}{t}\int_M \vn{A}^2\gamma^4d\mu
 + (2-\theta-\sqrt{\epsilon_0})
 \int_M \big(\vn{\nabla_{(2)}A}^2+\vn{A}^2\vn{\nabla A}^2+\vn{A}^6\big)\gamma^4d\mu
\\*
&\qquad\le c\big(1 + C^\frac{1}{3}_{A\!B} + C^\frac{1}{3}_{A\!B}\vn{h}^2_{\infty,[0,T^*]}
             + \epsilon_0^\frac{23}{6}+\epsilon_0^\frac{4}{3}\big)\epsilon^\frac{2}{3}
\\*
&\qquad\le c\epsilon^\frac{2}{3}.
\end{align*}
For $\theta$, $\epsilon_0$ small enough
we have
\[
  \rD{}{t}\int_{M} \vn{A}^2 \gamma^4d\mu 
+ \int_{M}
             \big(\vn{\nabla_{(2)}A}^2 + \vn{A}^2\vn{\nabla A}^2 + \vn{A}^6
             \big)\gamma^4 d\mu
\le c\epsilon^\frac{2}{3}.
\]
Integrating,
\begin{align*}
\int_{[\gamma=1]} \vn{A}^2 \gamma^4d\mu\ +
 &\int_0^t\int_{[\gamma=1]} (\vn{\nabla_{(2)}A}^2 +
\vn{A}^2\vn{\nabla A}^2 + \vn{A}^6) d\mu d\tau\\
&\le \int_{[\gamma>0]}\vn{A}^2d\mu\bigg|_{t=0}
 + c\epsilon^\frac{2}{3},
\end{align*}
where we used the fact $[\gamma=1]\subset[\gamma>0]$ and $0\le\gamma\le1$, with 
\[c = c(\epsilon_0, \vn{h}_{\infty,[0,t^*]}, c_{\gamma1},
c_{\gamma2},C_{A\!B}).\]
\end{proof}

\begin{rmk}
The assumption \eqref{AB} required for the three dimensional case is due to the fact that $L^2$
norms naturally arise when computing the evolution equations of various integral quantities, see the
proof of Corollary \ref{CorRE2} and Proposition \ref{p:prop3}.  Forcing $L^3$ norms in these
inequalities for the purpose of the above proof introduces changes in the exponents of the
$P$-terms, and to deal with this one would need to prove an altered form of Lemma \ref{MS1lem}.
This altered form will still require \eqref{AB} to handle the different exponents in the integrals.
So it seems to us that for the three dimensional case it is not possible to avoid assuming
\eqref{AB}, which is required to obtain results for non-trivial constraint functions regardless (see
Theorem \ref{thmapriorih}).
\end{rmk}

It remains only to prove the estimate used in the contradiction branch of the argument used to prove
the lifespan theorem.  
For this, we need some interpolation inequalities, and a preliminary
proposition.  We will only state the required interpolation inequality; the
proof can be found in \cite{kuwert2002gfw}.

\begin{prop}
\label{CORptermestimate}
Let $0\le i_1,\ldots,i_r\le k$, $i_1+\ldots+i_r = 2k$ and $s\ge 2k$.  Then for any tensor $T$
defined over an immersed hypersurface $f$ we have
\begin{equation*}
\int_M \nabla_{(i_1)}T*\cdots*\nabla_{(i_r)}T\gamma^sd\mu
  \le c\vn{T}^{r-2}_{\infty,[\gamma>0]}
       \left(\int_M\vn{\nabla_{(k)}T}^2\gamma^sd\mu + \vn{T}^2_{2,[\gamma>0]} \right).
\end{equation*}
\end{prop}


We now use this to derive the required proposition.

\begin{prop}
\label{p:prop45}
Suppose $f:M^n\times[0,T]\rightarrow\R^{n+1}$ is a \eqref{CSD} flow and
$\gamma:M\rightarrow\R$ a cutoff function as in \eqref{e:gamma}.  Then, for $s \ge 2k+4$ the
following estimate holds:
\begin{equation}
\label{e:prop45}
  \begin{split}
&\rD{}{t}\int_{M} \vn{\nabla_{(k)}A}^2\gamma^s d\mu
 + \int_{M} \vn{\nabla_{(k+2)}A}^2 \gamma^s d\mu \\
&\qquad\qquad\qquad 
  \le c\vn{A}^4_{\infty,[\gamma>0]}\int_M\vn{\nabla_{(k)}A}^2\gamma^sd\mu
       + c\vn{A}^2_{2,[\gamma>0]}(1+\vn{A}^4_{\infty,[\gamma>0]}) \\
&\qquad\qquad\qquad\qquad
       +ch\left(h^\frac{1}{3}\int_M\vn{\nabla_{(k)}A}^2\gamma^sd\mu
              + (1+h^\frac{1}{3})\vn{A}^2_{2,[\gamma>0]}\right).
  \end{split}
\end{equation}
\end{prop}
\begin{proof}
The proof is similar to \cite{kuwert2002gfw}, except for the terms which involve the constraint
function.  The nontrivial term is estimated as follows.  Let $r=3$ and $i_1+i_2=k$, $i_3=k$ in
Corollary \ref{CORptermestimate} to obtain
\begin{align}
h\int_M\left(\nabla_{(k)}[A*A]*\nabla_{(k)}A\right) \gamma^{s} d\mu
  &\le ch\sum_{\substack{i_1+i_2 = k\\0\le i_j\le k}}
        \int_M \nabla_{(i_1)}A*\nabla_{(i_2)}A*\nabla_{(i_3)}A\gamma^sd\mu
\notag\\
  &\le ch\vn{A}_\infty\left(\int_M\vn{\nabla_{(k)}A}^2\gamma^sd\mu + \vn{A}^2_{2,[\gamma>0]}
                          \right)
\notag\\
  &\le c\vn{A}^4_\infty\left(\int_M\vn{\nabla_{(k)}A}^2\gamma^sd\mu + \vn{A}^2_{2,[\gamma>0]}
                          \right)
\notag\\
&\qquad
  + h^\frac{4}{3}\left(\int_M\vn{\nabla_{(k)}A}^2\gamma^sd\mu + \vn{A}^2_{2,[\gamma>0]}
                          \right),
\notag
\end{align}
using Young's inequality.
\end{proof}

We now finish this section with a proof of the higher derivatives of curvature estimate, which will
allow us to both bound the constraint function in balls other than the `special ball' (see Corollary
8) and perform the contradiction part of our overall argument used to prove the lifespan
theorem.

\begin{prop} 
\label{p:prop46}
Let $n\in\{2,3\}$.  Suppose $f:M^n\times[0,T^*]\rightarrow\R^{n+1}$ is a \eqref{CSD} flow with $h$
satisfying \eqref{A1} and $\gamma$ as in \eqref{e:gamma}.  If $n=3$ assume in addition \eqref{AB}.
Then there is an $\epsilon_0$ depending on the constants in \eqref{e:gamma} and
$\vn{h}_{\infty,[0,T^*]}$ such that if
\begin{equation}
\label{eq13}
\sup_{[0,T^*]}\int_{[\gamma>0]}\vn{A}^nd\mu\le\epsilon_0,
\end{equation}
we can conclude
\begin{equation}
\label{eq14}
\vn{\nabla_{(k)}A}^2_{\infty,[\gamma=1]}
 \le c\big(k,T^*,c_{\gamma1},c_{\gamma2},\vn{h}_{\infty,[0,T^*]},\alpha_0(k+2),C_{A\!B}\big)
.
\end{equation}
\end{prop}
\begin{proof}
The idea is to use our previous estimates and then integrate.  
We fix $\gamma$ and consider special, tailor-made cutoff functions $\gamma_{\sigma,\tau}$ which
will allow us to combine our previous estimates.  Define for $0\le\sigma<\tau\le 1$ functions
$\gamma_{\sigma,\tau} = \psi_{\sigma,\tau}\circ\gamma$ satisfying $\gamma_{\sigma,\tau}=0$ for
$\gamma\le\sigma$ and $\gamma_{\sigma,\tau}=1$ for $\gamma\ge\tau$.  The function
$\psi_{\sigma,\tau}$ is chosen such that $\gamma_{\sigma,\tau}$ satisfies equation \eqref{e:gamma},
although with different constants.  Acceptable choices are
\[
 c_{\gamma_{\sigma,\tau}1} = \vn{\nabla\psi_{\sigma,\tau}}_\infty\cdot c_{\gamma1},
\text{ and }
c_{\gamma_{\sigma,\tau}2}
 = \max\{\vn{\nabla_{(2)}\psi_{\sigma,\tau}}_\infty\cdot c_{\gamma1}^2,
        \vn{\nabla\psi_{\sigma,\tau}}_\infty\cdot c_{\gamma2}\}.
\]
Using the cutoff function $\gamma_{0,\sfrac{1}{2}}$ instead of $\gamma$ in Proposition
\ref{p:prop44} gives
\begin{align}
  \int_0^{T^*}\int_{[\gamma\ge\sfrac{1}{2}]}\vn{\nabla_{(2)}A}^2+\vn{A}^6d\mu d\tau
    &\le c\epsilon_0(1+T^*)
\label{GradEsteq2}
\intertext{for $n=2$ and}
\notag
  \int_0^{T^*}\int_{[\gamma\ge\sfrac{1}{2}]}\vn{\nabla_{(2)}A}^2+\vn{A}^6d\mu d\tau
    &\le c\epsilon_0^\frac{2}{3}(C_{A\!B}^\frac{1}{3}+T^*)
\end{align}
for $n=3$.
Next, using $\gamma_{\sfrac{1}{2},\sfrac{3}{4}}$ in $\eqref{MS2}$ and equation \eqref{GradEsteq2}
above we obtain for $n=2$
\begin{equation}
\int_0^T\vn{A}^4_{\infty,[\gamma\ge\sfrac{3}{4}]}d\tau
  \le c\epsilon_0(c\epsilon_0(1+T^*)+\epsilon_0T^*)
  \le c\epsilon_0.
\label{GradEsteq3}
\end{equation}
For $n=3$ we have
\begin{equation}
\int_0^T\vn{A}^4_{\infty,[\gamma\ge\sfrac{3}{4}]}d\tau
  \le c(C_{A\!B})^\frac{1}{3}\epsilon_0^\frac{2}{3}
        \Big(2[c\epsilon_0^\frac{2}{3}(C_{A\!B}^\frac{1}{3}+T^*)]^\frac{3}{2}
           + c\epsilon_0(C_{A\!B})^\frac{1}{2}(T^*)^\frac{3}{2}\Big)
  \le c\epsilon_0,
\label{GradEsteq3new}
\end{equation}
where 
$c=c\big(\vn{h}_\infty, c_{\gamma1}, c_{\gamma2}, T^*, n, \epsilon_0
\big)$ for $n=2$ and
$c=c\big(\vn{h}_\infty, c_{\gamma1}, c_{\gamma2}, T^*, n, \epsilon_0,
C_{A\!B}\big)$ for $n=3$.
  We use the convention that for the remainder of this proof all constants $c$ will
depend on these quantities for $n=2$ and $n=3$ respectively.


We now use \eqref{e:prop45} with $\gamma_{\sfrac{3}{4},\sfrac{7}{8}}$.  Factorising, we have
\begin{align*}
\rD{}{t}\int_{M} \vn{\nabla_{(k)}A}^2\gamma_{\sfrac{3}{4},\sfrac{7}{8}}^s d\mu
&
  \le c\vn{A}^4_{\infty,[\gamma_{\sfrac{3}{4},\sfrac{7}{8}}\ge0]}
       \int_M\vn{\nabla_{(k)}A}^2\gamma_{\sfrac{3}{4},\sfrac{7}{8}}^sd\mu
\\
&\qquad
       + c\vn{A}^2_{2,[\gamma_{\sfrac{3}{4},\sfrac{7}{8}}\ge0]}
          \Big(1+\vn{A}^4_{\infty,[\gamma_{\sfrac{3}{4},\sfrac{7}{8}}\ge0]}\Big)
\\
&\qquad
       +ch\Big(h^\frac{1}{3}\int_M\vn{\nabla_{(k)}A}^2\gamma_{\sfrac{3}{4},\sfrac{7}{8}}^sd\mu
              + (1+h^\frac{1}{3})\vn{A}^2_{2,[\gamma_{\sfrac{3}{4},\sfrac{7}{8}}\ge0]}\Big)
\\
&
  \le c\Big(\vn{A}^4_{\infty,[\gamma\ge\sfrac{3}{4}]}+h^\frac{4}{3}\Big)
       \int_M\vn{\nabla_{(k)}A}^2\gamma_{\sfrac{3}{4},\sfrac{7}{8}}^sd\mu
\\
&\qquad
       + c\vn{A}^2_{2,[\gamma\ge\sfrac{3}{4}]}
         \Big(1+\vn{A}^4_{\infty,[\gamma\ge\sfrac{3}{4}]}+h+h^{\frac{4}{3}}\Big).
\end{align*}
%
Noting that the relevant integral quantities are bounded, we apply Gronwall's inequality and obtain
\begin{equation*}
  \int_{[\gamma\ge\sfrac{7}{8}]} \vn{\nabla_{(k)}A}^2d\mu
\le \beta(t) + \int_0^t \beta(\tau)\lambda(\tau)e^{\int_\tau^t\lambda(\nu)d\nu}d\tau
\le c\big(k,\alpha_0(k)\big)
,
\end{equation*}
where
\begin{align*}
\beta(t) &= \int_M \vn{\nabla_{(k)}A}^2\gamma_{\sfrac{3}{4},\sfrac{7}{8}}^sd\mu\bigg|_{t=0}
     + c\int_0^t
          \left[\vn{A}^2_{2,[\gamma\ge\sfrac{3}{4}]}
                \Big(1+\vn{A}^4_{\infty,[\gamma\ge\sfrac{3}{4}]}+h+h^{\frac{4}{3}}\Big)\right]d\tau,
\\
\lambda(t) &= 
        \vn{A}^4_{\infty,[\gamma\ge\sfrac{3}{4}]}+h^{\frac{4}{3}}.
\end{align*}
Trivially, we also have
\begin{equation*}
  \int_{[\gamma\ge\sfrac{7}{8}]} \vn{\nabla_{(k+2)}A}^2d\mu
\le c\big(k+2,\alpha_0(k+2)\big).
\end{equation*}
Therefore using \eqref{MS2secondstatement} with $\gamma_{\sfrac{7}{8},\sfrac{15}{16}}$ we can
bound $\vn{A}_\infty$ on a smaller ball:
\[
\vn{A}^{8n-12}_{\infty,[\gamma\ge\sfrac{15}{16}]}
\le c\epsilon_0\Big(\big[c(2,\alpha_0(2)\big)\big]^{\frac{2n^2-3n}{2}}+\epsilon_0\Big).
\]
Finally, using \eqref{MS2} with $T=\nabla_{(k)}A$ and $\gamma=\gamma_{\sfrac{15}{16},1}$ we 
obtain
\begin{align*}
\vn{\nabla_{(k)}A}_{\infty,[\gamma=1]}^4
 &\le c\vn{\nabla_{(k)}A}^{4-n}_{2,[\gamma>\sfrac{15}{16}]}
       \Big( \vn{\nabla_{(k+2)}A}^n_{2,[\gamma>\sfrac{15}{16}]}\\
&\hskip+3.5cm
                                   + (\vn{A}^{2n}_{\infty,[\gamma>\sfrac{15}{16}]}+1)
                                     \vn{\nabla_{(k)}A}^n_{2,[\gamma>\sfrac{15}{16}]}
       \Big)\\
 &\le c\big(k,\alpha_0(k+2)\big).
\end{align*}
This completes the proof of the proposition.
\end{proof}


\section{Proof of the lifespan theorem.}

Rescaling $\tilde{f}(x,t) = f(\frac{x}{\rho}, \frac{t}{\rho^4})$, the scale invariance of
\[
\int_{f^{-1}(B_\rho)} \vn{A}^n d\mu
\]
implies that we need only prove the theorem for $\rho=1$.
%
  We will show that 
\[
\tilde{T} \ge \frac{1}{c},
\]
and so scaling back we will conclude
inequality \eqref{eq2}.  

We make the definition
\begin{equation}
\label{e:epsfuncdef}
\eta(t) = \sup_{x\in\R^{n+1}}\int_{f^{-1}(B_1(x))} \vn{A}^nd\mu.
\end{equation}
By covering $B_1$ with several translated copies of $B_{\sfrac{1}{2}}$ there is a constant
$c_{\eta}$ such that
\begin{equation}
\label{e:epscovered}
\eta(t) \le c_{\eta}\sup_{x\in\R^{n+1}}\int_{f^{-1}(B_{\sfrac{1}{2}}(x))} \vn{A}^nd\mu.
\end{equation}
%
By short time existence we have that $f(M\times[0,t])$ is compact for $t<T$ and so the function
$\eta:[0,T)\rightarrow\R$ is continuous.  We now define
\begin{equation}
\label{e:struweparameter}
t^{(n)}_0
=
\begin{cases}
  \sup\{0\le t\le\min(T,\lambda_2) : \eta(\tau)\le 3c_{\eta}\epsilon_0
                                         \ \text{ for }\ 0\le\tau\le t\},
&\text{ $n=2$,}
\\
  \sup\{0\le t\le\min(T,\lambda_3) : \eta(\tau)\le 3c_{P22}c_\eta C^{1/3}_{A\!B}\epsilon_0^{2/3}
                                         \ \text{ for }\ 0\le\tau\le t\},
&\text{ $n=3$,}
\end{cases}
\end{equation}
where $\lambda_n$ is a parameter to be specified later.  
Recall that we assume \eqref{AB} in the case where $n=3$.  
The constant $c_{P22}$ is the maximum of 1 and the constant from Proposition \ref{p:prop46} with
$k=0$.
Note that the $\epsilon_0$ on the right hand side of the inequality is from equation \eqref{eq1}.

The proof continues in three steps.  First, we show that it must be the case that $t_0^{(n)} =
\min(T,\lambda_n)$.  Second, we show that if $t_0^{(n)} = \lambda_n$, then we can conclude the lifespan
theorem.  Finally, we prove by contradiction that if $T \ne \infty$, then $t_0^{(n)} \ne T$.  We
label these steps as 
\begin{align}
\label{e:7}
t_0^{(n)} &= \min(T,\lambda_n),\\
\label{e:8}
t_0^{(n)} &= \lambda_n \quad\Longrightarrow\quad\text{lifespan theorem},\\
\label{e:9}
T &\ne \infty\hskip+4.5mm \Longrightarrow\quad t_0^{(n)} \ne T.
\end{align}
The three statements \eqref{e:7}, \eqref{e:8}, \eqref{e:9} together imply the lifespan theorem.
%
We now give the proof of the first step, statement \eqref{e:7}.

From the assumption \eqref{eq1}, 
\[
\eta(0)\le\epsilon_0
 <
\begin{cases}
3c_{\eta}\epsilon_0,&\text{ for $n=2$}
\\
3c_{P22}c_\eta C^{1/3}_{A\!B}\epsilon_0^{2/3},&\text{ for $n=3$},
\end{cases}
\]
and therefore \eqref{e:struweparameter} implies $t^{(n)}_0 > 0$.  Assume for the sake of
contradiction that $t^{(n)}_0 < \min(T,\lambda_n)$.  Then from the definition
\eqref{e:struweparameter} of $t_0^{(n)}$ and the continuity of $\eta$ we have
\begin{equation}
\label{e:t0ltmin}
\eta\big(t_0^{(n)}\big) = 
\begin{cases}
3c_{\eta}\epsilon_0,&\text{ for $n=2$}
\\
3c_{P22}c_\eta C^{1/3}_{A\!B}\epsilon_0^{2/3},&\text{ for $n=3$,}
\end{cases}
\end{equation}
so long as $\epsilon_0 \le 1$ and $C_{A\!B},c_{P22} \ge 1$.
Recall Proposition \ref{p:prop44}.  We will now set $\gamma$ to be a cutoff function as in
\eqref{e:gamma} such that
\[
\chi_{B_{\sfrac{1}{2}}(x)} \le \tilde{\gamma} \le \chi_{B_1(x)},
\]
for any $x\in M_t$.
Choosing a small enough $\epsilon_0$ (by varying $\rho$ in \eqref{eq1}), definition
\eqref{e:struweparameter} implies that the smallness condition \eqref{eq9} is satisfied on
$[0,t_0^{(n)})$.
Due to our assumption \eqref{A1}, we also have that $\vn{h}_{\infty,[0,t_0^{(n)})} <
\infty$.  Therefore we have satisfied all the requirements of Proposition \ref{p:prop44}
, and so we conclude
\begin{equation}
\label{e:useprop44}
  \begin{split}
  \int_{f^{-1}(B_{\sfrac{1}{2}}(x))} \vn{A}^2 d\mu 
         &\le \int_{f^{-1}(B_1(x))} \vn{A}^2d\mu\bigg|_{t=0} + c_0c_{\eta}\epsilon_0^\frac{2}{n} t \\
         &\le 
\begin{cases}
2\epsilon_0,
&\text{for $n=2$ and }\lambda_2 = \frac{1}{c_0c_{\eta}},
\\
2c_{P22}C^{1/3}_{A\!B}\epsilon_0^{2/3},
&\text{for $n=3$ and }\lambda_3 = c_{P22}\frac{C^{1/3}_{A\!B}}{c_0c_{\eta}},
\end{cases}
  \end{split}
\end{equation}
for all $t \in [0,t^*]$, where $t^* < t_0^{(n)}$ and $c_0$ is the constant from Proposition
\ref{p:prop44}.  That is, equation \eqref{e:useprop44} above is true for all $t \in
\big[0,t_0^{(n)}\big)$.  We combine this with \eqref{e:epscovered} and Proposition \ref{p:prop46} to
conclude
\begin{equation}
\label{e:prop44pluscovering}
\eta(t)
 \le c^{n-2}_{P22}c_{\eta} \sup_{x\in\R^{n+1}} \int_{f^{-1}(B_{\sfrac{1}{2}}(x))} \vn{A}^2 d\mu 
 \le 
\begin{cases}
2c_{\eta}\epsilon_0,
&\text{ for $n=2$}
\\
2c_{P22}c_{\eta}C^{1/3}_{A\!B}\epsilon_0^{2/3},
&\text{ for $n=3$,}
\end{cases}
\end{equation}
where $0\le t < t_0^{(n)}$.

Since $\eta$ is continuous, we can let $t\rightarrow t_0^{(n)}$ and obtain a contradiction with
\eqref{e:t0ltmin}.  Therefore, with the choice of $\lambda_n$ in equation \eqref{e:useprop44}, the
assumption that $t_0^{(n)} < \min(T,\lambda_n)$ is incorrect.
Thus we have shown \eqref{e:7}, the first of our three steps.

We in fact have also proved the second step \eqref{e:8}. Observe that if $t_0^{(n)} =
\lambda_n$ then by the definition \eqref{e:struweparameter} of $t_0^{(n)}$, 
\[
  T\ge\lambda_n,
\]
which is \eqref{eq2}.  Also, \eqref{e:prop44pluscovering} implies \eqref{eq3}.  That is,
we have proved if $t_0^{(n)} = \lambda_n$, then the lifespan theorem holds, which is the second step
\eqref{e:8}.  It only remains to prove equation \eqref{e:9}.

We assume
\[
t_0^{(n)}=T\ne\infty;
\]
since if $T=\infty$ then \eqref{eq2} holds automatically and again \eqref{e:prop44pluscovering}
implies \eqref{eq3}.  Note also that we can safely assume $T < \lambda_n$, since otherwise
we can apply step two to conclude the lifespan theorem.

%

Since $T < \lambda_n$, \eqref{A1}
infers the existence of a $\varsigma > 0$ such that
\[
\vn{h}_{\infty,[0,T+\varsigma)} < \infty,
\]
which is enough (in terms of the constraint function) for short time existence to begin again at
time $T$.  To show that we may also extend the immersion $f$ to a time interval $[0,T+\varsigma)$,
we use Proposition \ref{p:prop46} and follow a standard proof such as that found in
\cite{kuwert2002gfw} or \cite{huisken1984fmc}.  Therefore we can extend the flow, contradicting the
maximality of $T$.

This establishes \eqref{e:9} and the theorem is proved.
\qed

\section{Concluding remarks.}

As mentioned earlier, Kuwert \& Sch\"atzle \cite{kuwert2002gfw} proved a lifespan theorem for the
Willmore flow,
\[
\pD{}{t}f = \big(\Delta H + Q(A)\big)\nu,
\]
where they considered surfaces immersed in $\R^n$ via $f$, i.e. $f:M^2\rightarrow\R^n$.  Note that
in one codimension $Q(A) =  \vn{A^o}^2H$.  We remark that one may consider the evolution equation
\[
\pD{}{t}f = \big(\Delta H + \tilde{Q}(A)\big)\nu,
\]
where $f:M^2\rightarrow\R^3$, with $\tilde{Q}(A)$ a term which may be estimated as
\begin{equation}
\tilde{Q} \le P_3^0(A)
\label{concremeq1}
\end{equation}
and recover a lifespan theorem.  One may employ exactly the techniques in \cite{kuwert2002gfw} to
obtain this result.  This is essentially due to the integral estimates not depending on the precise
form of the $P$-style terms.  It may be possible to improve the growth condition \eqref{concremeq1}
above to include some derivatives and more copies of $A$, however we have not pursued this.
  Of course combining this remark with the analysis we present in this paper for constrained
flows will give a lifespan theorem for flows of the form
\[
\pD{}{t} = \big(\Delta H + P_3^0(A) + h\big)\nu.
\]
Apart from constrained Willlmore flows (for which one may compute constraint functions which give
monotone area, volume, etc) we are not aware of any interesting examples of such flows.  For
immersions of dimension greater than 3, one will still be restricted by the Sobolev inequality
Theorem \ref{myLZthm}, and the local version Proposition \ref{MS2prop}.  We are not aware of any
technique which may be used to completely remove this restriction.

\section*{Acknowledgements}

This work forms part of the corresponding author's PhD thesis under the support of an Australian
Postgraduate Award.

Parts of this research was completed during two visits by the corresponding author to the Freie
Universit\"at in Berlin.  The first visit was under the support of the Deutscher Akademischer
Austausch Dienst, and the second supported by the Research Group in Geometric Analysis at the Freie
Universit\"at.
He is grateful for their support and hospitality.
The corresponding author would also like to thank Prof. Dr. Kuwert for useful discussions on this
work.

The first author was partly supported by an Australian Postdoctoral Fellowship as part of a
Discovery Grant from the Australian Research Council entitled ``Singularities and Surgery in
Curvature Flows''.  Parts of this work were completed while the first author was visiting the
University of Queensland and the Australian National University.  He is grateful for their support
and hospitality.

\bibliographystyle{spmpsci}
\bibliography{mbib}

\end{document}